\documentclass[a4paper,10pt]{article}

\usepackage{amsmath,amssymb,amsthm,amscd}
\usepackage[mathscr]{eucal}
\usepackage[all]{xy}
\usepackage{graphicx}

\makeatletter
  
  \@addtoreset{equation}{section}
\makeatother

\newcommand{\plim}{\varprojlim}
\newcommand{\mcal}{\mathcal}
\newcommand{\mbf}{\mathbf}

\newcommand{\mfrak}{\mathfrak}
\newcommand{\mbb}{\mathbb}
\newcommand{\mrm}{\mathrm}
\newcommand{\mfS}{\mathfrak{S}}
\newcommand{\vphi}{\varphi}

\newcommand{\mfM}{\mathfrak{M}}
\newcommand{\mfN}{\mathfrak{N}}

\newcommand{\mft}{\mathfrak{t}}

\newcommand{\whR}{\widehat{\mathcal{R}}}

\newcommand{\Mod}{\mrm{Mod}}

\newcommand{\cO}{\mathcal{O}}
\newcommand{\e}{\varepsilon}

\newtheorem{theorem}{Theorem}[section]
\newtheorem{corollary}[theorem]{Corollary}
\newtheorem{lemma}[theorem]{Lemma}
\newtheorem{proposition}[theorem]{Proposition}

\theoremstyle{definition}
\newtheorem{definition}[theorem]{Definition}
\newtheorem{remark}[theorem]{Remark}

\newtheorem{question}[theorem]{Question}

\newtheorem*{acknowledgments}{Acknowledgments}

\title{Lattices in potentially semi-stable representations and weak $(\vphi,\hat{G})$-modules}

\author{Yoshiyasu Ozeki\footnote{
Research Institute for Mathematical Sciences, Kyoto University,
Kyoto 606-8502, JAPAN.
\endgraf
e-mail: {\tt yozeki@kurims.kyoto-u.ac.jp}
\endgraf
Supported by JSPS KAKENHI Grant Number 25$\cdot$173}}

\date{}
\begin{document}
\maketitle

\begin{abstract}
Let $p$ be a prime number and $r\ge 0$ an integer.
In this paper, 
we prove that there exists an anti-equivalence 
between the category of weak $(\vphi,\hat{G})$-modules of height $\le r$
and a certain subcategory of the category of Galois stable $\mbb{Z}_p$-lattices in 
potentially semi-stable representations with Hodge-Tate weights
in $[0,r]$.
This gives an answer to a Tong Liu's question 
about the essential image of a functor on weak $(\vphi,\hat{G})$-modules.
For a proof, following Liu's methods,
we  construct linear algebraic data which classify 
lattices in potentially semi-stable representations.
\end{abstract}

%%%%%%%%%%%%%%%%%%%%%%%%%%%%%%%%%%%%%%%%%%%%%%%%%%%%%%%%%%%%%%%%%%%%%%%%%%%%%%%%%%%%%%%%%%%%%%%%%%%%%%%%%%%
%%%%%%%%%%%%%%%%%%%%%%%%%%%%%%%%%%%%%%%%%%%%%%%%%%%%%%%%%%%%%%%%%%%%%%%%%%%%%%%%%%%%%%%%%%%%%%%%%%%%%%%%%%%
%                           1                              %%%%%%%%%%%%%%%%%%%%%%%%%%%%%%%%%%%%%%%%%%%%%%%%
%%%%%%%%%%%%%%%%%%%%%%%%%%%%%%%%%%%%%%%%%%%%%%%%%%%%%%%%%%%%%%%%%%%%%%%%%%%%%%%%%%%%%%%%%%%%%%%%%%%%%%%%%%%
%%%%%%%%%%%%%%%%%%%%%%%%%%%%%%%%%%%%%%%%%%%%%%%%%%%%%%%%%%%%%%%%%%%%%%%%%%%%%%%%%%%%%%%%%%%%%%%%%%%%%%%%%%%

\section{Introduction}

Let $K$ be a complete discrete valuation field of mixed characteristics $(0,p)$ 
with perfect residue field. 
We take  a system of $p$-power roots $(\pi_n)_{n\ge 0}$  
of a uniformizer $\pi$ of $K$
such that $\pi_0=\pi$ and $\pi^p_{n+1}=\pi_n$.
We denote by $G_K$ and $G_{K_n}$
absolute Galois groups of $K$ and $K_n:=K(\pi_n)$, respectively. 

For applications to interesting problems such as modularity liftings,
it is useful to study an integral version of Fontaine's $p$-adic Hodge theory,
which is called integral $p$-adic Hodge theory.  
It is important in  integral $p$-adic Hodge theory 
to construct ``good'' linear algebraic data
which classify 
$G_K$-stable $\mbb{Z}_p$-lattices
in semi-stable, or crystalline, $\mbb{Q}_p$-representations of $G_K$
with Hodge-Tate weights in $[0,r]$.
Nowadays various such linear algebraic data are constructed; 
for example, so called Fontaine-Laffaille modules, 
Wach modules and Breuil modules.
It is one of the obstructions for the use of these algebraic data that 
we can not use them without restrictions
on the absolute ramification index $e$ of $K$ and (or) $r$.
In \cite{Li3}, 
based on a Kisin's insight \cite{Ki} for 
a classification of lattices in semi-stable representations,
Tong Liu defined notions of  
$(\vphi,\hat{G})$-modules
and weak
$(\vphi,\hat{G})$-modules.
He constructed 
a contravariant fully faithful functor 
$\hat{T}$ from the category 
of weak 
$(\vphi,\hat{G})$-modules of height $\le r$
into the the category of free $\mbb{Z}_p$-representations of $G_K$.
It is the main theorem of {\it loc.\ cit.}
that, without any restriction on $e$ and $r$, 
$\hat{T}$ induces an anti-equivalence between 
the category of $(\vphi,\hat{G})$-modules of height $\le r$ and 
the category of lattices in semi-stable $\mbb{Q}_p$-representations of $G_K$
with Hodge-Tate weights in $[0,r]$.
In the end of {\it loc. cit.},
he posed the following question:
\begin{question}
\label{question}
What is the essential image of the functor $\hat{T}$ on  {\it weak}
$(\vphi,\hat{G})$-modules?
\end{question} 

\noindent 
He showed that, if a representation of $G_K$ corresponds to 
a weak 
$(\vphi,\hat{G})$-module of height $\le r$, then
it is semi-stable over $K_n$ for some $n\ge 0$
and has Hodge-Tate weights in $[0,r]$.
However,  the converse does not hold in general. 
%(cf.\ \cite[Example 4.2.3]{Li3}).

In this paper, we give an answer to Question \ref{question}.
Denote by $m_0$ the maximum integer 
such that $K$  
contains $p^{m_0}$-th roots of unity.
For any non-negative integer $n$,
we denote by $\mcal{C}^r_n$ 
the category  of free $\mbb{Z}_p$-representations $T$ of $G_K$
which satisfy the following property;
there exists a semi-stable $\mbb{Q}_p$-representation $V$ of $G_K$
with Hodge-Tate weights in $[0,r]$
such that 
$T\otimes_{\mbb{Z}_p}\mbb{Q}_p$ is isomorphic to $V$ as representations of $G_{K_n}$.
Our main result is as follows.

\begin{theorem}
\label{Main1}
The essential image of 
the functor $\hat{T}$ is $\mcal{C}^r_{m_0}$.
\end{theorem}

\noindent
Therefore, 
we conclude that $\hat{T}$ induces an anti-equivalence between 
the category of weak 
$(\vphi,\hat{G})$-modules of height $\le r$
and the category $\mcal{C}^r_{m_0}$.

The crucial part of our proof is to show the relation 
$$
\mcal{C}^r_{m_0}\subset \mcal{C}^r\subset \mcal{C}^r_m
$$
where 
$\mcal{C}^r$ is
the essential image of 
the functor $\hat{T}$ and 
$m$ is the maximum integer 
such that the maximal unramified extension of $K$  
contains $p^m$-th roots of unity (cf.\ Lemma \ref{Lem:Main1'}).
We have two keys for our proof of this statement.
The first one is Proposition \ref{Main2},
which gives a relation between weak $(\vphi,\hat{G})$-modules 
and ``finite height'' representations.
For the proof, following the method of Liu's arguments of \cite{Li3} and \cite{Li4},
we construct certain linear data which classifies 
lattices in potentially semi-stable representations.
This is a direct generalization of the main result of \cite{Li3} 
(the idea for our proof is essentially due to Liu's previous works).
The second one is Proposition  \ref{Thm2};
it says
that the $G_{K_n}$-action of a finite height representation of $G_K$ 
which is semi-stable over $K_n$
extends to a $G_K$-action which is semi-stable over $K$.\\

\begin{acknowledgments}
The author thanks Akio Tamagawa 
who gave him useful advice in the proof of Lemma \ref{lastlemma}
in the case where $p$ is odd and $m_0=0$.
This work is supported by JSPS KAKENHI Grant Number 25$\cdot$173.
\end{acknowledgments}

\noindent
{\bf Notation :}
For any topological group $H$,
a free $\mbb{Z}_p$-representation of $H$ 
(resp.\  a $\mbb{Q}_p$-representation of $H$)
is a finitely generated free $\mbb{Z}_p$-module equipped 
with a continuous $\mbb{Z}_p$-linear $H$-action
(resp.\ a finite dimensional $\mbb{Q}_p$-vector space equipped 
with a continuous $\mbb{Q}_p$-linear $H$-action).
We denote by $\mrm{Rep}_{\mbb{Z}_p}(H)$
(resp.\ $\mrm{Rep}_{\mbb{Q}_p}(H)$)
the category of them.
For any field $F$, we denote by $G_F$ the absolute Galois group of $F$
(for a fixed separable closure of $F$).

%%%%%%%%%%%%%%%%%%%%%%%%%%%%%%%%%%%%%%%%%%%%%%%%%%%%%%%%%%%%%%%%%%%%%%%%%%%%%%%%%%%%%%%%%%%%%%%%%%%%%%%%%%%
%%%%%%%%%%%%%%%%%%%%%%%%%%%%%%%%%%%%%%%%%%%%%%%%%%%%%%%%%%%%%%%%%%%%%%%%%%%%%%%%%%%%%%%%%%%%%%%%%%%%%%%%%%%
%                           2                              %%%%%%%%%%%%%%%%%%%%%%%%%%%%%%%%%%%%%%%%%%%%%%%%
%%%%%%%%%%%%%%%%%%%%%%%%%%%%%%%%%%%%%%%%%%%%%%%%%%%%%%%%%%%%%%%%%%%%%%%%%%%%%%%%%%%%%%%%%%%%%%%%%%%%%%%%%%%
%%%%%%%%%%%%%%%%%%%%%%%%%%%%%%%%%%%%%%%%%%%%%%%%%%%%%%%%%%%%%%%%%%%%%%%%%%%%%%%%%%%%%%%%%%%%%%%%%%%%%%%%%%%

\section{Preliminary}

In this section, we recall some results on Liu's $(\vphi,\hat{G})$-modules
and related topics.
Throughout this paper, let $p\ge 2$ be a prime number.
Let $K$ be a complete discrete valuation field of mixed characteristics $(0,p)$ 
with perfect residue field $k$.
Let $L$ be a finite extension of $K$.
Take a uniformizer $\pi_L$ of $L$ and
a system of $p$-power roots $(\pi_{L,n})_{n\ge 0}$  
of $\pi_L$ such that $\pi_{L,0}=\pi_L$ and 
$\pi^p_{L,n+1}=\pi_{L,n}$.
We denote by $k_L$ the residue field of $L$. 
Put $L_n=L(\pi_{L,n}), L_{\infty}=\cup_{n\ge 0}L_n$ 
and define $\hat{L}$ to be the Galois closure of $L_{\infty}$ 
over $L$. 
We denote by $H_L$ and $\hat{G}_L$ the Galois group of $\hat{L}/L_{\infty}$ and $\hat{L}/L$,
respectively.
We denote by $K^{\mrm{ur}}$ and $L^{\mrm{ur}}$ maximal unramified extensions of $K$ and $L$, respectively.
Note that we have $L^{\mrm{ur}}=LK^{\mrm{ur}}$. 

Let $R=\plim \cO_{\overline{K}}/p$, 
where $\cO_{\overline{K}}$ is 
the integer ring of $\overline{K}$
and the transition maps are 
given by the $p$-th power map.
We write
$\underline{\pi_L}:=(\pi_{L,n})_{n\ge 0}\in R$.
Let $\mfS_{L}:=W(k_L)[\![u_L]\!]$ be the formal power series ring with indeterminate $u_L$.
We define a Frobenius endomorphism $\vphi$ of $\mfS_L$ by $u_L \mapsto u_L^p$
extending the Frobenius of $W(k_L)$.
The $W(k_L)$-algebra embedding $W(k_L)[u_L]\hookrightarrow W(R)$
defined by $u_L\mapsto [\underline{\pi_L}]$
extends to $\mfS_L\hookrightarrow W(R)$
where $[\ast]$ is the Teichm\"uller representative.

We denote by $\mrm{Mod}^r_{/\mfS_L}$ the category of 
$\vphi$-modules $\mfM$ over $\mfS_L$ which satisfy the following:
\begin{itemize}
\item $\mfM$ is free of finite type over $\mfS_L$; and 
\item $\mfM$ is of height $\le r$
in the sense that
$\mrm{coker}(1\otimes \vphi\colon \mfS_L\otimes_{\vphi,\mfS_L}\mfM\to \mfM)$ is killed by $E_L(u_L)^r$.
\end{itemize}

\noindent
Here,  $E_L(u_L)$ is the minimal polynomial of $\pi_L$ over $W(k_L)[1/p]$,
which is an Eisenstein polynomial.
We call objects of this category {\it Kisin modules of height $\le r$ over $\mfS_L$}.
We define a contravariant functor 
$T_{\mfS_L}\colon \mrm{Mod}^r_{/\mfS_L}\to \mrm{Rep}_{\mbb{Z}_p}(G_{L_{\infty}})$
by
$$
T_{\mfS_L}(\mfM):=
\mrm{Hom}_{\mfS_L,\vphi}(\mfM, W(R))
$$
\noindent
for an object $\mfM$ of  $\mrm{Mod}^r_{/\mfS_L}$.
Here a $G_{L_\infty}$-action on 
$T_{\mfS_L}(\mfM)$ is given by 
$(\sigma.g)(x)=\sigma(g(x))$ 
for $\sigma\in G_{L_\infty}, g\in T_{\mfS_L}(\mfM), x\in \mfM$.

\begin{proposition}[{\cite[Corollary 2.1.4 and Proposition 2.1.12]{Ki}}]
\label{Kisin}
The functor 
$T_{\mfS_L}\colon \mrm{Mod}^r_{/\mfS_L}\to \mrm{Rep}_{\mbb{Z}_p}(G_{L_{\infty}})$
is exact and fully faithful.
\end{proposition}

Let $S_L$ be the $p$-adic completion of $W(k_L)[u_L, \frac{E_L(u_L)^i}{i!}]_{i\ge 0}$
and endow $S_L$  with the following 
structures:
\begin{itemize}
\item a continuous $\vphi_{W(k_L)}$-semilinear Frobenius $\vphi\colon S_L\to S_L$ 
      defined by $u_L\mapsto u_L^p$. 
\item a continuous $W(k_L)$-linear derivation  $N\colon S_L\to S_L$
      defined by $N(u_L)=-u_L$.
\item a decreasing filtration 
      $(\mrm{Fil}^iS_L)_{i\ge 0}$ on $S_L$. 
      Here $\mrm{Fil}^iS_L$ 
      is the $p$-adic closure of the ideal generated 
      by $\frac{E_L(u_L)^j}{j!}$ for all $j\ge i$.       
\end{itemize}
The embedding $\mfS_L\hookrightarrow W(R)$
defined above extends to
$\mfS_L\hookrightarrow S_L\hookrightarrow A_{\mrm{cris}}$
and $S_L[1/p]\hookrightarrow B^+_{\mrm{cris}}$.
We take 
a primitive $p$-power root $\zeta_{p^n}$ of unity for $n\ge 0$
such that $\zeta^p_{p^{n+1}}=\zeta_{p^n}$.
We set 
$\underline{\e}:=(\zeta_{p^n})_{n\ge 0}\in R$
and $t:=-\mrm{log}([\underline{\e}])\in A_{\mrm{cris}}$.
For any integer $n\ge 0$,
let $t^{\{n\}}:=t^{r(n)}\gamma_{\tilde{q}(n)}(\frac{t^{p-1}}{p})$ 
where $n=(p-1)\tilde{q}(n)+r(n)$ with $\tilde{q}(n)\ge 0,\ 0\le r(n) <p-1$
and $\gamma_i(x)=\frac{x^i}{i!}$ 
the standard divided power.
Now we denote by $\nu\colon W(R)\to W(\overline{k})$ 
a unique lift of the projection $R\to \overline{k}$,
which extends to a map 
$\nu \colon B^+_{\mrm{cris}}\to W(\overline{k})[1/p]$.
For any subring $A\subset B^+_{\mrm{cris}}$,
we put 
$I_+A=\mrm{Ker}(\nu\ \mrm{on}\  B^+_{\mrm{cris}})\cap A$.

We define a subring $\mcal{R}_L$, containing $S_L$,
of $B^+_{\mrm{cris}}$ 
as below:
$$
\mcal{R}_L:=\left\{\sum^{\infty}_{i=0} f_it^{\{i\}}\mid f_i\in S_L[1/p]\
\mrm{and}\ f_i\to 0\ \mrm{as}\ i\to \infty\right\}.
$$
Furthermore, we define $\whR_L:=\mcal{R}_L\cap W(R)$.
We see that $S_L$ is not $G_L$-stable under the action of $G_L$ in $B^+_{\mrm{cris}}$.
However, $\mcal{R}_L, \whR_L, I_+\mcal{R}_L$ and $I_+\whR_L$ are $G_L$-stable.
Furthermore, they 
are stable under Frobenius in $B^+_{\mrm{cris}}$.
By definition $G_L$-actions on them
factor through $\hat{G}_L$.

For an object $\mfM$ of  $\mrm{Mod}^r_{/\mfS_L}$,
the map $\mfM\to \whR_L\otimes_{\vphi,\mfS_L} \mfM$
defined by $x\mapsto 1\otimes x$
is injective.
By this injection,
we often regard $\mfM$ as a $\vphi(\mfS_L)$-stable submodule of 
$\whR_L\otimes_{\vphi,\mfS_L} \mfM$.
\begin{definition}
\label{def:Liu}
A {\it weak $(\vphi,\hat{G}_L)$-module of height $\le r$ over $\mfS_L$}
is a triple $\hat{\mfM}=(\mfM,\vphi,\hat{G}_L)$ where
\begin{itemize}
\item[(1)] $(\mfM,\vphi)$ is an object of  $\mrm{Mod}^r_{/\mfS_L}$, 
\item[(2)] $\hat{G}_L$ is an $\whR_L$-semilinear continuous $\hat{G}_L$-action on 
           $\whR_L\otimes_{\vphi,\mfS_L}\mfM$,
\item[(3)] the $\hat{G}_L$-action commutes with $\vphi_{\whR_L}\otimes \vphi_{\mfM}$, and
\item[(4)] $\mfM\subset (\whR_L\otimes_{\vphi,\mfS_L} \mfM)^{H_L}$.
\end{itemize}
Furthermore, we say that $\hat{\mfM}$ is 
a {\it $(\vphi,\hat{G}_L)$-module of height $\le r$ over $\mfS_L$}
if $\hat{\mfM}$ satisfies the additional condition;
\begin{itemize}
\item[(5)] $\hat{G}_L$ acts on $\whR_L\otimes_{\vphi,\mfS_L} \mfM/I_+\whR_L(\whR_L\otimes_{\vphi,\mfS_L} \mfM)$  
           trivially.
\end{itemize}
We always regard $\whR_L\otimes_{\vphi,\mfS_L} \mfM$ as a 
$G_L$-module  via the projection $G_L\twoheadrightarrow \hat{G}_L$.
We denote by ${}_{\mrm{w}}\Mod^{r,\hat{G}_L}_{/\mfS_L}$ 
(resp.\ $\Mod^{r,\hat{G}_L}_{/\mfS_L}$)
the category of 
weak $(\vphi,\hat{G}_L)$-modules of height $\le r$ over $\mfS_L$
(resp.\ the category of  $(\vphi,\hat{G}_L)$-modules of height $\le r$ over $\mfS_L$).
\end{definition}

\noindent
We define a contravariant functor 
$\hat{T}_L\colon {}_{\mrm{w}}\Mod^{r,\hat{G}_L}_{/\mfS_L}\to 
\mrm{Rep}_{\mbb{Z}_p}(G_L)$
by
$$
\hat{T}_L(\hat{\mfM})=
\mrm{Hom}_{\whR_L,\vphi}(\whR_L\otimes_{\vphi, \mfS_L} \mfM, W(R)) 
$$
for an object $\hat{\mfM}=(\mfM,\vphi,\hat{G}_L)$ of ${}_{\mrm{w}}\Mod^{r,\hat{G}_L}_{/\mfS_L}$.
Here a $G_L$-action on 
$\hat{T}_L(\hat{\mfM})$ is given by 
$(\sigma.g)(x)=\sigma(g(\sigma^{-1}x))$ 
for $\sigma\in G_L, g\in \hat{T}_L(\hat{\mfM}), 
x\in \whR_L\otimes_{\vphi, \mfS_L} \mfM$.

\begin{remark}
\label{importantremark}
We should remark that 
notations $L_n,\mfS_L, \whR_L, \Mod^{r,\hat{G}_L}_{/\mfS_L},\dots $ above
{\it depend on the choices of a uniformizer $\pi_L$ of $L$ and a system $(\pi_{L,n})_{n\ge 0}$
of $p$-power roots of $\pi_L$}.
Conversely,
if we fix the choice of $\pi_L$ and $(\pi_{L,n})_{n\ge 0}$,
such notations are uniquely determined.
\end{remark}

\begin{theorem}
\label{Thm:Liu}
(1) (\cite[Theorem 2.3.1 (1)]{Li3})
Let $\hat{\mfM}=(\mfM,\vphi,\hat{G}_L)$ be an object of 
${}_{\mrm{w}}\Mod^{r,\hat{G}_L}_{/\mfS_L}$.
Then the map
$$
\theta\colon T_{\mfS_L}(\mfM) \to \hat{T}_L(\hat{\mfM})
$$
defined by $\theta(f)(a\otimes x):=a\vphi(f(x))$ 
for $a\in \whR_L$ and $x\in \mfM$,
is an isomorphism of representations of $G_{L_\infty}$.

\noindent
(2) (\cite[Theorem 2.3.1(2)]{Li3})
The contravariant functor $\hat{T}_L$ gives an anti-equivalence 
between the following categories:
\begin{itemize}
\item[--] The category of $(\vphi,\hat{G}_L)$-modules of height $\le r$ over $\mfS_L$.
\item[--] The category of $G_L$-stable $\mbb{Z}_p$-lattices in semi-stable 
$\mbb{Q}_p$-representations with 
Hodge-Tate weights in $[0,r]$.
\end{itemize}

\noindent
(3) (\cite[Theorem 4.2.2]{Li3})
The contravariant functor 
$\hat{T}_L\colon {}_{\mrm{w}}\Mod^{r,\hat{G}_L}_{/\mfS_L}\to 
\mrm{Rep}_{\mbb{Z}_p}(G_L)$
is fully faithful.
Furthermore, its essential image is 
contained in the category 
of $G_L$-stable $\mbb{Z}_p$-lattices 
in potentially semi-stable $\mbb{Q}_p$-representations of $G_L$
which are semi-stable over $L_n$ for some $n\ge 0$
and have Hodge-Tate weights in $[0,r]$.
\end{theorem}

\begin{remark}
\label{Rem:Liu}
Put $m=\mrm{max}\{i\ge 0 ; \zeta_{p^i}\in L^{\mrm{ur}}\}$.
We claim that any $\mbb{Q}_p$-representation of $G_L$ which is 
semi-stable over $L_n$ for some $n\ge 0$
is always semi-stable over $L_m$.

In the former half part of the proof of \cite[Theorem 4.2.2]{Li3},
a proof of this claim 
with ``$m=\mrm{max}\{i\ge 0 ; \zeta_{p^i}\in L\}$'' is written.
Unfortunately, there is a gap in the proof.
In the proof,
the assumption that the extension 
$L(\zeta_n,\pi_{L,n})/L$
is totally ramified is implicitly used
(p.\ $133$,\ between $l.\ 14$ and $l.\ 21$ of \cite{Li3}).
However, this condition is not satisfied in general.
So we need a little modification.
Put $m=\mrm{max}\{i\ge 0 ; \zeta_{p^i}\in L^{\mrm{ur}}\}$ as the beginning.
Denote by $\widehat{L^{\mrm{ur}}}$ the completion of 
$L^{\mrm{ur}}$.
We remark that the completion of 
the maximal unramified extension of $L_n$ is just
$\widehat{L^{\mrm{ur}}}(\pi_{L,n})$.
Let $V$ be a $\mbb{Q}_p$-representation of $G_L$ which is 
semi-stable over $L_n$ for some $n\ge 0$.
Then $V$ is semi-stable over  $\widehat{L^{\mrm{ur}}}(\pi_{L,n})$.
We remark that the proof of \cite[Theorem 4.2.2]{Li3} exactly holds 
at least under the assumption that the residue field of 
the base field is algebraically closed. (We need only the first paragraph of {\it loc.\ cit.} here.)
Thus we know that $V$ is semi-stable over $\widehat{L^{\mrm{ur}}}(\pi_{L,m})$
and thus we obtain the claim.
\end{remark}

Now we restate Theorem \ref{Main1} with the above setting of notation
and give a  further result.
Fix the choice of a uniformizer $\pi_K$ of $K$ 
and a system $(\pi_{K,n})_{n\ge 0}$ of $p$-power roots of $\pi_K$,
and define notations $K_n, \Mod^{r,\hat{G}_K}_{/\mfS_K},\dots $
with respect to them.
Recall that $m_0$ (resp.\ $m$) is the maximum integer 
such that $K$  
(resp.\ $K^{\mrm{ur}}$)
contains $p^{m_0}$-th (resp.\ $p^m$-th) roots of unity.
We note that the inequality  $m_0\le m$ always holds.
For any non-negative integer $n$,
we denote by $\mcal{C}^r_n$ 
the category  of free $\mbb{Z}_p$-representations $T$ of $G_K$
which satisfy the following property;
there exists a semi-stable $\mbb{Q}_p$-representation $V$ of $G_K$
with Hodge-Tate weights in $[0,r]$
such that 
$T\otimes_{\mbb{Z}_p}\mbb{Q}_p$ is isomorphic to $V$ as representations of $G_{K_n}$.

Our goal in this paper is to show the following:

\begin{theorem}
\label{Main1'}
The essential image of the functor 
$\hat{T}_K\colon {}_{\mrm{w}}\Mod^{r,\hat{G}_K}_{/\mfS_K}\to \mrm{Rep}_{\mbb{Z}_p}(G_K)$
is $\mcal{C}^r_{m_0}$.
\end{theorem}

As an immediate  consequence of the above theorem, we obtain
\begin{corollary}
The functor $\hat{T}_K$ induces an anti-equivalence 
${}_{\mrm{w}}\Mod^{r,\hat{G}_K}_{/\mfS_K}\overset{\sim}{\to} \mcal{C}^r_{m_0}$.
\end{corollary}

For later use,
we end this section by describing the following proposition.

\begin{proposition}
\label{totst}
Let $L$ be a finite totally ramified extension of $K$.
Then the restriction functor from the category of 
semi-stable $\mbb{Q}_p$-representations of $G_K$
into the category of 
semi-stable $\mbb{Q}_p$-representations of $G_L$ 
is fully faithful.
\end{proposition}
\begin{proof}
In view of the theory of Fontaine's filtered $(\vphi, N)$-modules, 
the result immediately follows from calculations of elementary 
linear algebras.
\end{proof}

%%%%%%%%%%%%%%%%%%%%%%%%%%%%%%%%%%%%%%%%%%%%%%%%%%%%%%%%%%%%%%%%%%%%%%%%%%%%%%%%%%%%%%%%%%%%%%%%%%%%%%%%%%%
%%%%%%%%%%%%%%%%%%%%%%%%%%%%%%%%%%%%%%%%%%%%%%%%%%%%%%%%%%%%%%%%%%%%%%%%%%%%%%%%%%%%%%%%%%%%%%%%%%%%%%%%%%%
%                           3                             %%%%%%%%%%%%%%%%%%%%%%%%%%%%%%%%%%%%%%%%%%%%%%%%
%%%%%%%%%%%%%%%%%%%%%%%%%%%%%%%%%%%%%%%%%%%%%%%%%%%%%%%%%%%%%%%%%%%%%%%%%%%%%%%%%%%%%%%%%%%%%%%%%%%%%%%%%%%
%%%%%%%%%%%%%%%%%%%%%%%%%%%%%%%%%%%%%%%%%%%%%%%%%%%%%%%%%%%%%%%%%%%%%%%%%%%%%%%%%%%%%%%%%%%%%%%%%%%%%%%%%%%

\section{Proof of Main Theorem}

Our main goal in this section is to give a proof of Theorem \ref{Main1'}.
In the first three subsections,
we prove the following lemma, which plays an important role 
in our proof. 

\begin{lemma}
\label{Lem:Main1'}
Denote by $\mcal{C}^r$
the essential image of 
$\hat{T}_K\colon {}_{\mrm{w}}\Mod^{r,\hat{G}_K}_{/\mfS_K}\to \mrm{Rep}_{\mbb{Z}_p}(G_K)$.
Then we have $\mcal{C}^r_{m_0}\subset \mcal{C}^r \subset \mcal{C}^r_m$.
\end{lemma}

Clearly, Theorem \ref{Main1'} follows immediately from the lemma if $m_0=m$.
However, the condition $m_0=m$ is not always satisfied.
Before starting a main part of this section, we give some remarks about this condition.

\begin{proposition}
\label{m0=m}
(1) If $k$ is algebraically closed, then $m_0=m$.

\noindent
(2) If $K(\zeta_{p^{m_0+1}})/K$ is ramified, then $m_0=m$.

\noindent
(3) Suppose that $\zeta_p\in K$ (resp.\ $\zeta_4\in K$) if $p$ is odd (resp.\ $p=2$). 
    Then $\hat{K}$ is totally ramified over $K$ 
    if and only if $m_0=m$.
\end{proposition} 
\begin{proof}
The assertion (1) and (2) is clear.
We prove (3).
If $m_0 < m$, then $K(\zeta_{p^m})$ is a non-trivial unramified extension of $K$
and thus the extension $\hat{K}/K$ is not totally ramified.
Conversely, suppose that $\hat{K}/K$ is not totally ramified.
Then there exists an integer $n\ge 0$ such that 
$K(\zeta_{p^n},\pi_n)/K$ is not  totally ramified.
This implies so is $K(\zeta_{p^n},\pi_n)/K(\pi_n)$.
We may suppose $n\ge m$.
Since $\mrm{Gal}(K(\zeta_{p^n},\pi_n)/K(\pi_n))$
is isomorphic to $\mbb{Z}/p^{n-m_0}\mbb{Z}$ (here we need the assumption 
$\zeta_p\in K$ (resp.\ $\zeta_4\in K$) if $p$ is odd (resp.\ $p=2$)),
any subfield of $K(\zeta_{p^n},\pi_n)/K(\pi_n)$
is of the form $K(\zeta_{p^l},\pi_n)$ for $m_0\le l\le n$.
Thus there exists an integer $m_0\le l_0\le n$
such that 
$K^{\mrm{ur}}(\pi_n)\cap K(\zeta_{p^n},\pi_n)=K(\zeta_{p^{l_0}},\pi_n)$.
We have
$\zeta_{p^{l_0}}\in K^{\mrm{ur}}(\pi_n)\cap K^{\mrm{ur}}(\zeta_{p^n})$.
Since $\zeta_p\in K$ (resp.\ $\zeta_4\in K$) if $p$ is odd (resp.\ $p=2$), 
we have also  $K^{\mrm{ur}}(\pi_n)\cap K^{\mrm{ur}}(\zeta_{p^n})=K^{\mrm{ur}}$.
This implies $l_0\le m$. 
Since the residue field extension corresponding to  
$K(\zeta_{p^n},\pi_n)/K(\pi_n)$
is non-trivial,
the extension $K(\zeta_{p^{l_0}},\pi_n)/K(\pi_n)$
is non-trivial extension
and thus so is
$K(\zeta_{p^m},\pi_n)/K(\pi_n)$. 
This implies 
$1<[K(\zeta_{p^m},\pi_n):K(\pi_n)]=[K(\zeta_{p^m}):K]$
and hence $m_0<m$.
\end{proof}

\begin{remark}
\label{rem:m0<m}
The condition $m_0=m$ is not always satisfied.
Here are some examples.

\noindent
(1) Suppose $p>2$.
Set $\alpha:=(2+p)^{1/(p-1)}, \beta:=(-p)^{1/(p-1)}$ and $K:=\mbb{Q}_p(\alpha \beta)$.
The field $K$ is totally ramified over $\mbb{Q}_p$ since 
the minimal polynomial of $\alpha\beta$ over $\mbb{Q}_p$ is 
an Eisenstein polynomial $X^{p-1}-(2+p)(-p)$. 
It is well-known that $\mbb{Q}_p(\beta)=\mbb{Q}_p(\zeta_p)$.
The extension $K(\zeta_p)/K$ is not totally ramified 
since so is $\mbb{Q}_p(\alpha)/\mbb{Q}_p$ and $p>2$
(note that the residue class of $\alpha$ is not contained in $\mbb{F}_p$). 
Now we take any odd prime $p$ such that the extension 
$\mbb{Q}(\alpha)/\mbb{Q}$ is unramified (e.g., $p=3,5,7,\dots $).
Then $K(\zeta_p)/K$ is an unramified extension.
This implies that $m_0=0 < m$.
(Moreover, we see that $m=1$.)

\noindent
(2) Suppose $p=2$ and set $K:=\mbb{Q}_2(\sqrt{-5})$.
Then $K(\zeta_4)/K$ is unramified extension of degree $2$, and thus 
$m_0=1<m$. (Moreover, we see that $m=2$.)

\noindent
(3) Let $K'$ be a finite extension of $\mbb{Q}_p$ such that 
it contains $p$-th roots of unity and $K'(\zeta_{p^{\infty}})/K'$
is a totally ramified extension.
Let $K''$ be an unramified $\mbb{Z}_p$-extension of $K'$.
We denote by $K'_{(n)}$ and $K''_{(n)}$ the unique degree-$p^n$-subextensions of 
$K'(\zeta_{p^{\infty}})/K'$ and $K''/K'$, respectively.
Explicitly, the field
$K'_{(n)}$ coincides with $K'(\zeta_{p^{m'_0+n}})$
where $m'_0=\max \{i\ge 0 \mid \zeta_{p^i}\in K'\}$.
If we denote by $M_{(n)}$ the composite field of  $K'_{(n)}$ and $K''_{(n)}$,
then we have  isomorphisms 
$$
\mrm{Gal}(M_{(n)}/K')\simeq 
\mrm{Gal}(K'_{(n)}/K')\times \mrm{Gal}(K''_{(n)}/K')  
\simeq \mbb{Z}/p^n\mbb{Z}\times \mbb{Z}/p^n\mbb{Z}
$$  
Let $K$ be the subfield of $M_{(n)}/K'$ which corresponds to 
the group of diagonal components of $\mrm{Gal}(M_{(n)}/K')\simeq 
\mbb{Z}/p^n\mbb{Z}\times \mbb{Z}/p^n\mbb{Z}$ 
via Galois theory. We consider $m_0$ and $m$ for this $K$. 
Since $K\cap K'_{(n)}=K'$, we know $m_0=m'_0$.
On the other hand,
since $M_{(n)}=KL_{(n)}=K(\zeta_{p^{m'_0+n}})$ and the extension $M_{(n)}/K$
is unramified, we have $m\ge m_0'+n=m_0+n$.
\end{remark}

%%%%%%%%%%%%%%%%%%%%%%%%%%%%%%%%%%%%%%%%%%%%%%%%%%%%%%%%%%%%%%%%%%%%%%%%%%%%%%%%%%%%%%%%%%%%%%%%%%%%%%%%%%%
%%%%%%%%%%%%%%%%%%%%%%%%%%%%%%%%%%%%%%%%%%%%%%%%%%%%%%%%%%%%%%%%%%%%%%%%%%%%%%%%%%%%%%%%%%%%%%%%%%%%%%%%%%%
%                           3.1                              %%%%%%%%%%%%%%%%%%%%%%%%%%%%%%%%%%%%%%%%%%%%%%%%
%%%%%%%%%%%%%%%%%%%%%%%%%%%%%%%%%%%%%%%%%%%%%%%%%%%%%%%%%%%%%%%%%%%%%%%%%%%%%%%%%%%%%%%%%%%%%%%%%%%%%%%%%%%
%%%%%%%%%%%%%%%%%%%%%%%%%%%%%%%%%%%%%%%%%%%%%%%%%%%%%%%%%%%%%%%%%%%%%%%%%%%%%%%%%%%%%%%%%%%%%%%%%%%%%%%%%%%

\subsection{Lattices in potentially semi-stable representations}
\label{3.1}

In this subsection 
we define a notion of $(\vphi,\hat{G}_L,K)$-modules
which classifies 
lattices in
potentially semi-stable $\mbb{Q}_p$-representations of $G_K$
which are semi-stable over $L$.

\begin{definition}
\label{def}
A {\it $(\vphi,\hat{G}_L,K)$-module of height $\le r$ over $\mfS_L$}
is a pair $(\hat{\mfM}, G_K)$ where 
\begin{itemize}
\item[(1)] $\hat{\mfM}=(\mfM,\vphi,\hat{G}_L)$ is an object of $\Mod^{r,\hat{G}_L}_{/\mfS_L}$, 
\item[(2)] $G_K$ is a $W(R)$-semilinear continuous $G_K$-action on 
           $W(R)\otimes_{\vphi,\mfS_L}\mfM$,
\item[(3)] the $G_K$-action commutes with $\vphi_{W(R)}\otimes \vphi_{\mfM}$, and
\item[(4)] the $W(R)$-semilinear $G_L$-action on 
           $W(R)\otimes_{\vphi,\mfS_L}\mfM (\simeq W(R)\otimes_{\whR_L} (\whR_L\otimes_{\vphi,\mfS_L}\mfM))$
           induced from the $\hat{G}_L$-structure of $\hat{\mfM}\in\Mod^{r,\hat{G}_L}_{/\mfS_L}$
           coincides with     
           the restriction of the $G_K$-action of (2) to $G_L$.
 
\end{itemize}
If $(\hat{\mfM}, G_K)$ is a $(\vphi,\hat{G}_L,K)$-module of height $\le r$ over $\mfS_L$,
we often abuse notations by writing $\hat{\mfM}$ for $(\hat{\mfM}, G_K)$ for simplicity.
We denote by $\Mod^{r,\hat{G}_L,K}_{/\mfS_L}$
the category of 
$(\vphi,\hat{G}_L,K)$-modules of height $\le r$ over $\mfS_L$.
\end{definition}
\noindent
We define a contravariant functor 
$\hat{T}_{L/K}\colon \Mod^{r,\hat{G}_L,K}_{/\mfS_L}\to 
\mrm{Rep}_{\mbb{Z}_p}(G_K)$
by
$$
\hat{T}_{L/K}(\hat{\mfM})=
\mrm{Hom}_{W(R),\vphi}(W(R)\otimes_{\vphi, \mfS_L} \mfM, W(R)) 
$$
for an object $\hat{\mfM}$ of $\Mod^{r,\hat{G}_L,K}_{/\mfS_L}$
with underlying Kisin module $\mfM$.
Here a $G_K$-action on 
$\hat{T}_{L/K}(\hat{\mfM})$ is given by 
$(\sigma.g)(x)=\sigma(g(\sigma^{-1}x))$ 
for $\sigma\in G_K, g\in \hat{T}_{L/K}(\hat{\mfM}), 
x\in W(R)\otimes_{\vphi, \mfS_L} \mfM$.
Note that we have natural isomorphisms 
\begin{align*}
\mrm{Hom}_{\whR_L,\vphi}(\whR_L\otimes_{\vphi, \mfS_L} \mfM, W(R))
&\overset{\sim}{\rightarrow}
\mrm{Hom}_{W(R),\vphi}(W(R)\otimes_{\vphi, \whR_L}(\whR_L\otimes_{\vphi, \mfS_L} \mfM), W(R))\\
&\overset{\sim}{\rightarrow}
\mrm{Hom}_{W(R),\vphi}(W(R)\otimes_{\vphi, \mfS_L}\mfM, W(R)).
\end{align*}
Thus we obtain 
\begin{equation}
\label{eta}
\eta\colon 
\hat{T}_L(\hat{\mfM})\overset{\sim}{\longrightarrow} \hat{T}_{L/K}(\hat{\mfM}).
\end{equation}
This is  $G_L$-equivariant by the condition (4) of Definition \ref{def}.
In particular, 
$\hat{T}_{L/K}(\hat{\mfM})\otimes_{\mbb{Z}_p} \mbb{Q}_p$ is semi-stable over $L$
by Theorem \ref{Thm:Liu} (2).

The goal of the rest of this subsection is to prove the following theorem.
\begin{theorem}
\label{thm1}
The contravariant functor $\hat{T}_{L/K}$ induces an anti-equivalence 
between the following categories:
\begin{itemize}
\item[--] The category of 
$(\vphi,\hat{G}_L,K)$-modules of height $\le r$ over $\mfS_L$.
\item[--] The category of $G_K$-stable $\mbb{Z}_p$-lattices 
in potentially semi-stable 
$\mbb{Q}_p$-representations of $G_K$ 
which are semi-stable over $L$
and have Hodge-Tate weights in $[0,r]$.
\end{itemize}
\end{theorem}
The above theorem follows by essentially the same arguments of  Liu 
(\cite{Li3}, \cite{Li4}), 
but  we write a proof here for the sake of completeness.
Before a proof, we recall Liu's comparison morphisms 
between $(\vphi,\hat{G}_L)$-modules and representations associated with them.
Furthermore, 
we define its variant for  $(\vphi,\hat{G}_L,K)$-modules.

Let $\hat{\mfM}=(\mfM,\vphi,\hat{G}_L)$ 
be a weak $(\vphi,\hat{G}_L)$-module of height $\le r$ over $\mfS_L$.
By identifying $\hat{T}_L(\hat{\mfM})$ with 
$\mrm{Hom}_{W(R),\vphi}(W(R)\otimes_{\vphi, \whR_L}(\whR_L\otimes_{\vphi, \mfS_L} \mfM), W(R))$,
we define a $W(R)$-linear map 
$$
\hat{\iota}_L\colon W(R)\otimes_{\whR_L} (\whR_L\otimes_{\vphi,\mfS_L} \mfM) \to
W(R)\otimes_{\mbb{Z}_p} \hat{T}^{\vee}_L(\hat{\mfM})
$$
by the composite 
$W(R)\otimes_{\whR_L} (\whR_L\otimes_{\vphi,\mfS_L} \mfM) 
\to
\mrm{Hom}_{\mbb{Z}_p}(\hat{T}_L(\hat{\mfM}),W(R))\overset{\sim}{\to}
W(R)\otimes_{\mbb{Z}_p} \hat{T}^{\vee}_L(\hat{\mfM})$.
Here, the first arrow is defined by $x\mapsto (f\mapsto f(x), \forall{f}\in \hat{T}_L(\hat{\mfM}))$
and the second is a natural isomorphism.
Also, for a $(\vphi,\hat{G}_L,K)$-module $\hat{\mfM}$ 
of height $\le r$ over $\mfS_L$,
we define a natural $W(R)$-linear map 
$$
\hat{\iota}_{L/K}\colon W(R)\otimes_{\vphi,\mfS_L}\mfM \hookrightarrow 
W(R)\otimes_{\mbb{Z}_p} \hat{T}^{\vee}_{L/K}(\hat{\mfM})
$$
by a similar way. Let $\mfrak{t}$ be an element of $W(R)\smallsetminus pW(R)$
such that $\vphi(\mfrak{t})=pE_L(u_L)E_L(0)^{-1}\mfrak{t}$.
Such $\mfrak{t}$ is unique up to units of $\mbb{Z}_p$.

\begin{proposition}
\label{comparison}
(1) (\cite[Proposition 3.1.3]{Li3})
The map
$
\hat{\iota}_L
$
as above is injective, which preserves Frobenius and $G_L$-actions.
Furthermore, we have
$\vphi(\mfrak{t})^r(W(R)\otimes_{\mbb{Z}_p} \hat{T}^{\vee}_L(\hat{\mfM}))
\subset \mrm{Im}\ \hat{\iota}_L$.

\noindent
(2) The map
$
\hat{\iota}_{L/K}
$ 
as above is injective, which preserves Frobenius and $G_K$-actions.
Furthermore, we have
$\vphi(\mfrak{t})^r(W(R)\otimes_{\mbb{Z}_p} \hat{T}^{\vee}_{L/K}(\hat{\mfM}))
\subset \mrm{Im}\ \hat{\iota}_{L/K}$.

\noindent
(3)  Let $\hat{\mfM}$
be a $(\vphi,\hat{G}_L,K)$-module of height $\le r$ over $\mfS_L$
with underlying Kisin module $\mfM$.
Then the following diagram is commutative:
\begin{center}
$
\displaystyle \xymatrix{
W(R)\otimes_{\whR_L} (\whR_L\otimes_{\vphi,\mfS_L} \mfM) 
\ar@{^{(}->}^{\hat{\iota}_L}[rr] \ar[d]_{\wr}
& &  
W(R)\otimes_{\mbb{Z}_p} \hat{T}^{\vee}_L(\hat{\mfM})  
\\
W(R)\otimes_{\vphi,\mfS_L}\mfM \ar@{^{(}->}^{\hat{\iota}_{L/K}}[rr]
& & 
W(R)\otimes_{\mbb{Z}_p} \hat{T}^{\vee}_{L/K}(\hat{\mfM}) 
\ar[u]_{W(R)\otimes \eta^{\vee}}^{\wr}
}$
\end{center}
Here, the left vertical arrow is a natural isomorphism and  
$\eta$ is defined in (\ref{eta}). 
\end{proposition}
\begin{proof}
The commutativity of (3) is clear by construction, and
the rest assertions follow by essentially the same proof as  \cite[Proposition 3.1.3]{Li3}.
\end{proof}

In the rest of this subsection,
we denote by $\mrm{Rep}^{r,L\mathchar`-\mrm{st}}_{\mbb{Z}_p}(G_K)$ the full subcategory of 
$\mrm{Rep}_{\mbb{Z}_p}(G_K)$ appeared in Theorem \ref{thm1}.
The isomorphism $\eta$ shows below.

\begin{lemma}
The functor $\hat{T}_{L/K}$ has values in $\mrm{Rep}^{r,L\mathchar`-\mrm{st}}_{\mbb{Z}_p}(G_K)$.
\end{lemma}

Next we show the fully faithfulness of the functor $\hat{T}_{L/K}$.

\begin{lemma}
The functor $\hat{T}_{L/K}$ is fully faithful.
\end{lemma}

\begin{proof}
Let $\hat{\mfM}$ and 
$\hat{\mfM}'$
be $(\vphi,\hat{G}_L,K)$-modules of height $\le r$ over $\mfS_L$
with underlying Kisin modules $\mfM$ and $\mfM'$, respectively.
Take any $G_K$-equivariant morphism
$f\colon \hat{T}_{L/K}(\hat{\mfM})\to \hat{T}_{L/K}(\hat{\mfM}')$.
By the map $\eta$, we identify $\hat{T}_{L/K}(\hat{\mfM})$ and $\hat{T}_{L/K}(\hat{\mfM}')$
with $\hat{T}_L(\hat{\mfM})$ and $\hat{T}_L(\hat{\mfM}')$, respectively.
Since $\hat{T}_L$ is fully faithful,
there exists a unique morphism 
$\mfrak{f}\colon \hat{\mfM}'\to \hat{\mfM}$
of $(\vphi,\hat{G}_L)$-modules of height $\le r$ over $\mfS_L$
such that $\hat{T}_L(\mfrak{f})=f$.
It is enough to show that 
$\mfrak{f}$ is in fact a morphism of $(\vphi,\hat{G}_L,K)$-modules,
that is, $W(R)\otimes \mfrak{f}\colon W(R)\otimes_{\vphi, \mfS_L}\mfM'\to W(R)\otimes_{\vphi, \mfS_L}\mfM$
is $G_K$-equivariant.
Consider the following diagram:
\begin{center}
$
\displaystyle \xymatrix{
W(R)\otimes_{\vphi,\mfS_L} \mfM 
\ar@{^{(}->}^{\hat{\iota}_{L/K}}[rr] 
& &  
W(R)\otimes_{\mbb{Z}_p} \hat{T}^{\vee}_{L/K}(\hat{\mfM})  
\\
W(R)\otimes_{\vphi,\mfS_L}\mfM' 
\ar@{^{(}->}^{\hat{\iota}_{L/K}}[rr]
\ar[u]^{W(R)\otimes \mfrak{f}}
& & 
W(R)\otimes_{\mbb{Z}_p} \hat{T}_{L/K}(\hat{\mfM}') 
\ar[u]^{W(R)\otimes f^{\vee}}
}$
\end{center} 
We see that the above diagram is commutative.
Since $W(R)\otimes f^{\vee}$ and two horizontal arrows
above are $G_K$-equivariant,
so is $W(R)\otimes \mfrak{f}$.
\end{proof}

\begin{lemma}
\label{ess1}
The functor $\hat{T}_{L/K}\colon \Mod^{r,\hat{G}_L,K}_{/\mfS_L}\to 
\mrm{Rep}^{r,L\mathchar`-\mrm{st}}_{\mbb{Z}_p}(G_K)$ 
is essentially surjective
if $L$ is a Galois extension of $K$.
\end{lemma}

To show this lemma, 
we recall arguments of \cite[\S 2]{Li4}. 
Suppose $L$ is a (not necessary totally ramified) Galois extension of $K$.
Let $T$ be an object of $\mrm{Rep}^{r,L\mathchar`-\mrm{st}}_{\mbb{Z}_p}(G_K)$.
Put $d=\mrm{rank}_{\mbb{Z}_p}T$.
Take a  $(\vphi,\hat{G}_L)$-module 
$\hat{\mfM}=(\mfM,\vphi,\hat{G}_L)$  over $\mfS_L$ 
such that $\hat{T}_L(\hat{\mfM})= T|_{G_L}$.
We consider the map
$\hat{\iota}_L\colon W(R)\otimes_{\vphi,\mfS_L}\mfM
\hookrightarrow W(R)\otimes_{\mbb{Z}_p} \hat{T}^{\vee}_L(\hat{\mfM})
=W(R)\otimes_{\mbb{Z}_p} T^{\vee}$.
By the same argument as the proof of \cite[Lemma 2.3.1]{Li4},
we can check the following
\begin{lemma}
\label{claim}
$W(R)\otimes_{\vphi,\mfS_L}\mfM$ is stable 
under the $G_K$-action via $\hat{\iota}_L$.
\end{lemma}

\noindent
We include (a main part of) the proof in {\it loc.\ cit.} of this lemma  here
since we will use this argument again in the next subsection
(cf.\ the proof of Theorem \ref{Main2}).
By \cite{Br}, we know that 
$\mcal{D}:=S_L[1/p]\otimes_{\vphi,\mfS_L}\mfM$
has a structure of a Breuil module\footnote{We do not describe the definition of
Breuil modules in this note. See \S 6.1 of \cite{Br} for axioms of Breuil modules.} 
which 
corresponds to $V|_{G_L}$, where
$V:=T\otimes_{\mbb{Z}_p}\mbb{Q}_p$.
In particular, we have a monodromy operator $N_{\mcal{D}}$ on $\mcal{D}$.
Set $D:=\mcal{D}/I_+S_L[1/p]\mcal{D}$.
There exists a unique $\vphi$-compatible $W(k_L)$-linear section
$s\colon D\hookrightarrow \mcal{D}$.
Breuil showed in {\it loc.\ cit.} that 
$N_{\mcal{D}}$ preserves $s(D)$
and thus
we can define $\tilde{N}:=N_{\mcal{D}}|_{s(D)}\colon s(D)\to s(D)$.
Then the $G_L$-action on 
$B^+_{\mrm{st}}\otimes_{S_L[1/p]}s(D)(=B^+_{\mrm{st}}\otimes_{\whR_L} (\whR_L\otimes_{\vphi,\mfS_L} \mfM))$
induced from the $\hat{G}_L$-structure on $\hat{\mfM}$
is given by 
$$
g(a\otimes x)=\sum^{\infty}_{i=0}g(a)\gamma_i(-\mrm{log}([\underline{\e}(g)]))
\otimes \tilde{N}^i(x)
$$
for any $g\in G_L, a\in B^+_{\mrm{st}}$ and $x\in s(D)$.
Here,  $\underline{\e}(g):=g(\underline{\pi_L})/\underline{\pi_L}\in R^{\times}$.
Set 
$$
\bar{D}:=\left\{\sum^{\infty}_{i=0} \gamma_i(\mfrak{u})\otimes \tilde{N}^i(x)
\mid x\in s(D) \right\}
\subset B^+_{\mrm{st}}\otimes_{W(k_L)[1/p]}s(D)
$$
where $\mfrak{u}:=\mrm{log}([\underline{\pi_L}])\in B^+_{\mrm{st}}$.
This is a $\vphi$-stable $W(k_L)[1/p]$-vector space of dimension $d$.
Setting the monodromy $N_{B^+_{\mrm{st}}}$
on $B^+_{\mrm{st}}$ by $N(\mfrak{u})=1$,
we equip $B^+_{\mrm{st}}\otimes_{W(k_L)[1/p]}s(D)$ 
(resp.\ $B^+_{\mrm{st}}\otimes_{\mbb{Q}_p}V^{\vee}$)
with a monodromy operator $N$ by $N:=N_{B^+_{\mrm{st}}}\otimes 1_{s(D)}$ 
(resp.\ $N:=N_{B^+_{\mrm{st}}}\otimes 1_{V^{\vee}}$).
Then it is easy to see that $\bar{D}$ is stable under $N$.
On the other hand,
we have a natural $G_K$-equivariant injection
$\iota\colon B^+_{\mrm{st}}\otimes_{W(k_L)[1/p]} D_{\mrm{st}}(V)
\hookrightarrow B^+_{\mrm{st}}\otimes_{\mbb{Q}_p} V^{\vee}
$
where $D_{\mrm{st}}(V):=(B^+_{\mrm{st}}\otimes_{\mbb{Q}_p} V^{\vee})^{G_L}$
is a filtered $(\vphi,N)$-module over $L$.
(Here we remark that $D_{\mrm{st}}(V)$ is equipped with a natural $G_K$-action
since $L/K$ is Galois.)
Since $G_L$ acts on $\bar{D}$
trivially (cf.\ \S 7.2 of \cite{Li1}), 
the image of $\bar{D}$ under the injection 
$
B^+_{\mrm{st}}\otimes_{W(k_L)[1/p]}s(D)=
B^+_{\mrm{st}}\otimes_{\whR_L} (\whR_L\otimes_{\vphi,\mfS_L} \mfM)
\overset{\hat{\iota}_{L,B}}{\hookrightarrow} 
B^+_{\mrm{st}}\otimes_{\mbb{Q}_p}V^{\vee} 
$
is equal to  $\iota(D_{\mrm{st}}(V))$.
Here, $\hat{\iota}_{L,B}:=B^+_{\mrm{st}}\otimes \hat{\iota}_L$,
which is compatible with Frobenius and monodromy operators.
Hence we have an isomorphism $\hat{i}\colon D_{\mrm{st}}(V)\overset{\sim}{\longrightarrow} \bar{D}$
which makes the following diagram commutative:
\begin{center}
$
\displaystyle \xymatrix{
D_{\mrm{st}}(V)
\ar_{\wr}^{\hat{i}}[d]
& \subset & 
B^+_{\mrm{st}}\otimes_{W(k_L)[1/p]} D_{\mrm{st}}(V) 
\ar@{^{(}->}^{\qquad \quad \iota}[rr] 
& &
B^+_{\mrm{st}}\otimes_{\mbb{Q}_p} V^{\vee} 
\ar@{=}[d]
\\
\bar{D}
& \subset &
B^+_{\mrm{st}}\otimes_{W(k_L)[1/p]}s(D)
\ar@{^{(}->}^{\qquad \hat{\iota}_{L,B}}[rr]
& &
B^+_{\mrm{st}}\otimes_{\mbb{Q}_p}V^{\vee} 
}$
\end{center} 
Note that $\hat{i}$ is compatible with Frobenius and monodromy operators.
We identify $D_{\mrm{st}}(V)$ with $\bar{D}$ by $\hat{i}$.

Let $e_1,\dots, e_d$ be a $W(k_L)[1/p]$-basis of $D$, and 
define a matrix $\bar{N}\in M_d(W(k_L)[1/p])$ by 
$\tilde{N}(s(e_1),\dots ,s(e_d))=(s(e_1),\dots ,s(e_d))\bar{N}$.
Put $\bar{e}_j=\sum^{\infty}_{i=0} \gamma_i(\mfrak{u})\otimes \tilde{N}^i(s(e_j))$
for any $j$.
Then $\bar{e}_1,\dots ,\bar{e}_d$ is a basis of $D_{\mrm{st}}(V)=\bar{D}$.
An easy calculation shows that
the monodromy $N$ on $D_{\mrm{st}}(V)=\bar{D}$ 
is represented by $\bar{N}$ for this basis, 
that is,
$N(\bar{e}_1,\dots ,\bar{e}_d)=(\bar{e}_1,\dots ,\bar{e}_d)\bar{N}$.
We define a matrix $A_g\in GL_d(W(k_L)[1/p])$
by 
$g(\bar{e}_1,\dots ,\bar{e}_d)
=(\bar{e}_1,\dots ,\bar{e}_d)A_g$
for any $g\in G_K$.
Since the $G_K$-action on $D_{\mrm{st}}(V)=\bar{D}$ is compatible with $N$,
we have the relation $A_gg(\bar{N})=\bar{N}A_g$.
Consequently, we have 
\begin{equation}
\label{action}
g(s(e_1),\dots ,s(e_d))=(s(e_1),\dots ,s(e_d))\mrm{exp}(-\lambda_g\bar{N})A_g
\end{equation}
in $B^+_{\mrm{st}}\otimes_{\mbb{Q}_p}V^{\vee}$, 
where $\lambda_g:=\mrm{log}([g(\underline{\pi_L})/\underline{\pi_L}])\in B^+_{\mrm{cris}}$.
This implies that 
$B^+_{\mrm{cris}}\otimes_{\vphi,\mfS_L}\mfM
=B^+_{\mrm{cris}}\otimes_{W(k_L)[1/p]} s(D)$ is stable 
under the $G_K$-action via $\hat{\iota}_{L,B}$.
Now Lemma \ref{claim} follows by 
an easy combination of Proposition \ref{comparison} (1) and \cite[Lemma 3.2.2]{Li3}
(cf.\ the first paragraph of the proof of \cite[Lemma 2.3.1]{Li4}).

\begin{proof}[Proof of Lemma \ref{ess1}]
We continue to use the same notation as above.
By Lemma \ref{claim}, we know that 
$\hat{\mfM}$
has a structure of an object of $\Mod^{r,\hat{G}_L,K}_{/\mfS_L}$ 
with the property that the map
$W(R)\otimes_{\vphi,\mfS_L}\mfM\overset{\hat{\iota}_L}{\hookrightarrow} 
W(R)\otimes_{\mbb{Z}_p} \hat{T}^{\vee}_L(\hat{\mfM})
=W(R)\otimes_{\mbb{Z}_p} T^{\vee}$
is $G_K$-equivariant.
Let $\eta\colon \hat{T}_L(\hat{\mfM})
\overset{\sim}{\longrightarrow} \hat{T}_{L/K}(\hat{\mfM})$
be the isomorphism defined in (\ref{eta}).
By Proposition \ref{comparison} (3),
we know that $W(R)\otimes \eta^{\vee}$ induces an isomorphism
$\hat{\iota}_L(W(R)\otimes_{\vphi,\mfS_L}\mfM)\overset{\sim}{\longrightarrow}
\hat{\iota}_{L/K}(W(R)\otimes_{\vphi,\mfS_L}\mfM)$,
which is $G_K$-equivariant.
Since $\vphi(\mft)^r(W(R)\otimes_{\mbb{Z}_p} \hat{T}^{\vee}_L(\hat{\mfM}))$
(resp.\ $\vphi(\mft)^r(W(R)\otimes_{\mbb{Z}_p} \hat{T}^{\vee}_{L/K}(\hat{\mfM}))$)
is contained in $(\hat{\iota}_L(W(R)\otimes_{\vphi,\mfS_L}\mfM))$
(resp.\ $(\hat{\iota}_{L/K}(W(R)\otimes_{\vphi,\mfS_L}\mfM))$),
we known that the map
$\vphi(\mft)^r(W(R)\otimes_{\mbb{Z}_p} \hat{T}^{\vee}_L(\hat{\mfM}))
\overset{\sim}{\longrightarrow} 
\vphi(\mft)^r(W(R)\otimes_{\mbb{Z}_p} \hat{T}^{\vee}_{L/K}(\hat{\mfM}))$
induced from $W(R)\otimes \eta^{\vee}$ is $G_K$-equivariant.
Thus so is $\eta\colon T=\hat{T}_L(\hat{\mfM})\overset{\sim}{\longrightarrow}
\hat{T}_{L/K}(\hat{\mfM})$.
\end{proof}

\begin{remark}
Let $\hat{e}_1,\dots ,\hat{e}_d$ be a $\mfS_K$-basis of $\vphi^{\ast}\mfM$,
which is also an  $S_K[1/p]$-basis of $\mcal{D}$.
Denote by $e_i$ the image  of $\hat{e}_i$ under the projection
$\mcal{D}\twoheadrightarrow  D$.
Then $e_1,\dots ,e_d$ is a $W(k)[1/p]$-basis of $D$. 
For these basis, we see that
the matrix $A_g\in GL_d(W(k_L)[1/p])$  as above
is in fact contained in $GL_d(W(k_L))$ by Lemma \ref{ess1}.
(However, we never use this fact in the present paper.)
\end{remark}

\begin{lemma}
The functor $\hat{T}_{L/K}\colon \Mod^{r,\hat{G}_L,K}_{/\mfS_L}\to 
\mrm{Rep}^{r,L\mathchar`-\mrm{st}}_{\mbb{Z}_p}(G_K)$ 
is essentially surjective
for any finite extension $L$ of $K$.
\end{lemma}

\begin{proof}
Let $T$ be an object of $\mrm{Rep}^{L\mathchar`-\mrm{st},r}_{\mbb{Z}_p}(G_K)$.
Let $L'$ be the Galois closure of $L$ over $K$
(and fix the choice of a uniformizer of $L'$ and a system of $p$-power roots of it;
see Remark \ref{importantremark}).
Since we have already shown Theorem \ref{thm1} for $\hat{T}_{L'/K}$,
we know that there exists a  
$(\vphi,\hat{G}_{L'},K)$-module $\hat{\mfM}'$
over $\mfS_{L'}$
such that $\hat{T}_{L'/K}(\hat{\mfM}')\simeq T$
as representations of $G_K$.
On the other hand, we have a unique $(\vphi,\hat{G}_L)$-module 
$\hat{\mfM}$ such that $T\simeq \hat{T}_L(\hat{\mfM})$
as representations of $G_L$ since $T$ is semi-stable over $L$.
We denote by $\mfM'$ and $\mfM$ 
underlying Kisin modules of $\hat{\mfM}'$ and $\hat{\mfM}$,
respectively.
By \cite[Theorem 3.2.1]{Li5} and Proposition \ref{comparison} (3),
the image of $W(R)\otimes_{\vphi,\mfS_{L'}} \mfM'$ under 
$\hat{\iota}_{L'/K}$ is equal to that of $W(R)\otimes_{\vphi,\mfS_L}\mfM$
under $\hat{\iota}_L$.
Hence we have a $\vphi$-equivariant isomorphism 
$W(R)\otimes_{\vphi,\mfS_{L'}} \mfM' \simeq W(R)\otimes_{\vphi,\mfS_L}\mfM$.
We define a $G_K$-action on $W(R)\otimes_{\vphi,\mfS_L}\mfM$
by this isomorphism. 
Then 
$\hat{\mfM}$ has a structure of 
$(\vphi,\hat{G}_L,K)$-module
over $\mfS_L$ so that
$\hat{\iota}_L\colon W(R)\otimes_{\vphi,\mfS_L}\mfM\hookrightarrow
W(R)\otimes_{\mbb{Z}_p} T^{\vee}$
is $G_K$-equivariant.
Since 
$\hat{T}_{L/K}(\hat{\mfM})
=\mrm{Hom}_{W(R),\vphi}(W(R)\otimes_{\vphi,\mfS_L}\mfM,W(R))
\simeq \mrm{Hom}_{W(R),\vphi}(W(R)\otimes_{\vphi,\mfS_{L'}}\mfM',W(R))
=\hat{T}_{L'/K}(\hat{\mfM}')=T$ 
as representations of $G_K$,
we have done.
\end{proof}

%%%%%%%%%%%%%%%%%%%%%%%%%%%%%%%%%%%%%%%%%%%%%%%%%%%%%%%%%%%%%%%%%%%%%%%%%%%%%%%%%%%%%%%%%%%%%%%%%%%%%%%%%%%
%%%%%%%%%%%%%%%%%%%%%%%%%%%%%%%%%%%%%%%%%%%%%%%%%%%%%%%%%%%%%%%%%%%%%%%%%%%%%%%%%%%%%%%%%%%%%%%%%%%%%%%%%%%
%                           3.2                              %%%%%%%%%%%%%%%%%%%%%%%%%%%%%%%%%%%%%%%%%%%%%%%%
%%%%%%%%%%%%%%%%%%%%%%%%%%%%%%%%%%%%%%%%%%%%%%%%%%%%%%%%%%%%%%%%%%%%%%%%%%%%%%%%%%%%%%%%%%%%%%%%%%%%%%%%%%%
%%%%%%%%%%%%%%%%%%%%%%%%%%%%%%%%%%%%%%%%%%%%%%%%%%%%%%%%%%%%%%%%%%%%%%%%%%%%%%%%%%%%%%%%%%%%%%%%%%%%%%%%%%%

\subsection{$\mcal{C}^r_{m_0}\subset \mcal{C}^r$}
\label{3.2}

We prove the relation
$\mcal{C}^r_{m_0}\subset \mcal{C}^r$ in the assertion of Lemma \ref{Lem:Main1'}.
At first, fix the choices of a uniformizer $\pi_K$ of $K$ 
and a system $(\pi_{K,n})_{n\ge 0}$ of $p$-power roots of $\pi_K$,
and define notations $K_n, \mfS_K, \Mod^{r,\hat{G}_K}_{/\mfS_K},\dots $
with respect to them (see also Remark \ref{importantremark}).
We also consider notations 
$\mfS_{K_n}, S_{K_n},\dots $
with respect to the uniformizer $\pi_{K_n}:=\pi_{K,n}$ of $K_n$ 
and the system $(\pi_{K,n+m})_{m\ge 0}$ of $p$-power roots of $\pi_{K_n}$.
Note that we have $\mfS_K\subset \mfS_{K_n}$, $S_K\subset S_{K_n}$ and $E_{K_n}(u_{K_n})=E_K(u_K)$
with the relation
$u^{p^n}_{K_n}=u_K$.

To show the relation
$\mcal{C}^r_{m_0}\subset \mcal{C}^r$,
it follows from Lemma 2.1.15 of \cite{Ki} 
that it suffices to show the following.
\begin{proposition}
\label{Main2}
Let  $T$ be a free $\mbb{Z}_p$-representation of $G_K$
which is semi-stable over $K_n$ for some $n\le m_0$
and $T|_{G_{K_{\infty}}}\simeq T_{\mfS_K}(\mfM)$ for some $\mfM\in  \mrm{Mod}^r_{/\mfS_K}$.
Then there exists a (unique)  
weak $(\vphi,\hat{G}_K)$-module $\hat{\mfM}$ of height $\le r$ over $\mfS_K$
such that $\hat{T}_K(\hat{\mfM})\simeq T$.
%\noindent
%(2) Assume $m_0=m$.
%Then the contravariant functor 
%$\hat{T}_K\colon {}_{\mrm{w}}\Mod^{r,\hat{G}_K}_{/\mfS_K}\to 
%\mrm{Rep}_{\mbb{Z}_p}(G_K)$
%induces an anti-equivalence between the following categories:
%\begin{itemize}
%\item[--] The category  of 
%weak $(\vphi,\hat{G}_K)$-modules of height $\le r$ over $\mfS_K$.
%\item[--] The category of free $\mbb{Z}_p$-representations $T$ of $G_K$
%which are semi-stable
%over $K_n$ for some $n\ge 0$ and $T|_{G_{K_{\infty}}}\simeq T_{\mfS_K}(\mfM)$ for some $\mfM\in  %\mrm{Mod}^r_{/\mfS_K}$.  
%(By Remark \ref{Rem:Liu}, we can always choose such $n$ to be $m$.)
%\end{itemize}
\end{proposition}

\begin{proof}
%By \cite[Theorem 4.2.2]{Li3}, it suffices to show (1).
Let $T, n$ and $\mfM$ be as in the statement.
Note that $K_n$ is a now Galois extension of $K$ for such $n$,
and note also that 
$\mfM_n:=\mfS_{K_n}\otimes_{\mfS_K}\mfM$ is a Kisin module of 
height $\le r$ over $\mfS_{K_n}$. 
By Theorem \ref{thm1},
there exists a $(\vphi,\hat{G}_{K_n},K)$-module $\hat{\mfN}$ over $\mfS_{K_n}$
such that $T\simeq \hat{T}_{K_n/K}(\hat{\mfN})$.
Denote by $\mfN$ the underlying Kisin module of $\hat{\mfN}$.
Since $T_{\mfS_{K_n}}(\mfM_n)$ is isomorphic to $T_{\mfS_{K_n}}(\mfN)$,
we may identify $\mfN$ with $\mfM_n$.
Thus $\mfM_n$ is equipped with a structure of a
$(\vphi,\hat{G}_{K_n},K)$-module $\hat{\mfM}_n$ over $\mfS_{K_n}$
such that $T\simeq \hat{T}_{K_n/K}(\hat{\mfM}_n)$.
Putting $\vphi^{\ast}\mfM=\mfS_K\otimes_{\vphi,\mfS_K}\mfM$,
we know that $G_K(\vphi^{\ast}\mfM)$ is contained in 
$W(R)\otimes_{\vphi,\mfS_{K_n}}\mfM_n=W(R)\otimes_{\vphi,\mfS_K}\mfM$.
We claim that $G_K(\vphi^{\ast}\mfM)$ is contained in $\mcal{R}_{K}\otimes_{\vphi,\mfS_K}\mfM$.
Admitting this claim, we see that $\mfM$ has a structure of a weak $(\vphi,\hat{G}_K)$-module
of height $\le r$ over $\mfS_K$
which corresponds to $T$, and hence we finish a proof.

Put $\mcal{D}_n=S_{K_n}[1/p]\otimes_{\vphi,\mfS_{K_n}}\mfM_n$ and
$\mcal{D}=S_K[1/p]\otimes_{\vphi,\mfS_K}\mfM$.
Let $\hat{e}_1,\dots ,\hat{e}_d$ be a $\mfS_K$-basis of $\vphi^{\ast}\mfM$,
which is an $S_{K_n}[1/p]$-basis of $\mcal{D}_n$ and an $S_K[1/p]$-basis of $\mcal{D}$.
Denote by $e_i$ the image  of $\hat{e}_i$ under the projection
$\mcal{D}\twoheadrightarrow  \mcal{D}/I_+S_K[1/p]=:D$.
Then $e_1,\dots ,e_d$ is a $W(k)[1/p]$-basis of $D$. 
By \cite[Proposition 6.2.1.1]{Br}, we have a unique $\vphi$-compatible 
section $s\colon D\hookrightarrow \mcal{D}$ of the projection 
$\mcal{D}\twoheadrightarrow D$.
Since $\mcal{D}=S_K[1/p]\otimes_{W(k)[1/p]} s(D)$,
there exists a matrix $X\in GL_d(S_K[1/p])$
such that 
$(\hat{e}_1,\dots ,\hat{e}_d)=(s(e_1),\dots ,s(e_d))X$.
Now we extend the $G_K$-action on $W(R)\otimes_{\vphi,\mfS_{K_n}} \mfM_n$
to $B^+_{\mrm{cris}}\otimes_{W(k)[1/p]} s(D)
=B^+_{\mrm{cris}}\otimes_{W(R)} (W(R)\otimes_{\vphi,\mfS_{K_n}} \mfM_n)$
by a natural way.
Take any $g\in G_K$ and put $\lambda_g=\mrm{log}([g(\underline{\pi_{K_n}})/\underline{\pi_{K_n}}])$.
We see that $\lambda_g$ is contained in $\mcal{R}_{K}$.
Recall that $K_n$ is now a totally ramified Galois extension over $K$. 
By (\ref{action}), 
we have $g(s(e_1),\dots ,s(e_d))=(s(e_1),\dots ,s(e_d))\mrm{exp}(-\lambda_g\bar{N})A_g$
for some nilpotent matrix $\bar{N}\in M_d(W(k)[1/p])$ and
some $A_g\in GL_d(W(k)[1/p])$. 
Therefore, we obtain 
$g(\hat{e}_1,\dots ,\hat{e}_d)
=(\hat{e}_1,\dots ,\hat{e}_d)X^{-1}\mrm{exp}(-\lambda_g\bar{N})A_gg(X)$.
Since the matrix $X^{-1}\mrm{exp}(-\lambda_g\bar{N})A_gg(X)$ has coefficients in 
$\mcal{R}_{K}$, we have done.
\end{proof}

\begin{remark}
We remark that, for any semi-stable $\mbb{Q}_p$-representation $V$ of $G_{K_n}$
with Hodge-Tate weights in $[0,r]$,
there exists a Kisin module $\mfM_n\in  \mrm{Mod}^r_{/\mfS_{K_n}}$ 
such that $V|_{G_{K_{\infty}}}$ is isomorphic to 
 $T_{\mfS_{K_n}}(\mfM_n)\otimes_{\mbb{Z}_p}\mbb{Q}_p$ (cf. \cite[Lemma 2.1.15]{Ki}).
The above proposition 
studies the case where 
$\mfM_n$  descends 
to a Kisin module over $\mfS_K$, but 
this condition is not always satisfied.
An example for this is given in the proof of Proposition \ref{prop:rem}.
%We remark that, for any semi-stable $\mbb{Q}_p$-representation $V$ of $G_{K_n}$
%with Hodge-Tate weights in $[0,r]$,
%there exists a Kisin module $\mfM_n\in  \mrm{Mod}^r_{/\mfS_{K_n}}$ 
%such that $V|_{G_{K_{\infty}}}$ is isomorphic to 
% $T_{\mfS_{K_n}}(\mfM_n)\otimes_{\mbb{Z}_p}\mbb{Q}_p$.
%However, such $\mfM$ does not always  descend 
%to Kisin modules over $\mfS_K$.
%An example for this is given 
%as follows:
%Suppose $n,m,r\ge 1$.
%By Proposition \ref{prop:rem}, given in the next subsection,
%there exists a free $\mbb{Z}_p$-representation $T$
%of $G_K$ such that it is semi-stable over $K_n$ and
%that the $G_{K_n}$-action on $T\otimes_{\mbb{Z}_p}\mbb{Q}_p$
%does not extend to a semi-stable representation of $G_K$.
%Define  $V:=(T\otimes_{\mbb{Z}_p}\mbb{Q}_p)|_{G_{K_n}}$.
%This is semi-stable over $K_n$.
%Assume that there exists a Kisin module $\mfM\in  \mrm{Mod}^r_{/\mfS_K}$ 
%such that $V|_{G_{K_{\infty}}}\simeq T_{\mfS_K}(\mfM)\otimes_{\mbb{Z}_p}\mbb{Q}_p$.
%Since $T$ is a $G_{K_{\infty}}$-stable $\mbb{Z}_p$-lattice of $V$,
%we can find a Kisin module $\mfN\in  \mrm{Mod}^r_{/\mfS_K}$
%such that $T|_{G_{K_{\infty}}}\simeq T_{\mfS_K}(\mfN)$ by \cite[Lemma 2.1.15]{Ki}.
%However, this contradicts Corollary \ref{Main3} below. 
\end{remark}

%%%%%%%%%%%%%%%%%%%%%%%%%%%%%%%%%%%%%%%%%%%%%%%%%%%%%%%%%%%%%%%%%%%%%%%%%%%%%%%%%%%%%%%%%%%%%%%%%%%%%%%%%%%
%%%%%%%%%%%%%%%%%%%%%%%%%%%%%%%%%%%%%%%%%%%%%%%%%%%%%%%%%%%%%%%%%%%%%%%%%%%%%%%%%%%%%%%%%%%%%%%%%%%%%%%%%%%
%                           3.3                              %%%%%%%%%%%%%%%%%%%%%%%%%%%%%%%%%%%%%%%%%%%%%%%%
%%%%%%%%%%%%%%%%%%%%%%%%%%%%%%%%%%%%%%%%%%%%%%%%%%%%%%%%%%%%%%%%%%%%%%%%%%%%%%%%%%%%%%%%%%%%%%%%%%%%%%%%%%%
%%%%%%%%%%%%%%%%%%%%%%%%%%%%%%%%%%%%%%%%%%%%%%%%%%%%%%%%%%%%%%%%%%%%%%%%%%%%%%%%%%%%%%%%%%%%%%%%%%%%%%%%%%%

\subsection{$\mcal{C}^r \subset \mcal{C}^r_m$}
\label{3.3}

Next we prove the relation $\mcal{C}^r \subset \mcal{C}^r_m$
in the assertion of Lemma \ref{Lem:Main1'}.
The key for our proof is the following proposition.

\begin{proposition}
\label{Thm2}
The restriction functor 
$\mrm{Rep}_{\mbb{Q}_p}(G_K)\to \mrm{Rep}_{\mbb{Q}_p}(G_{K_n})$
induces an equivalence between the following categories:
\begin{itemize}
\item[--] The category of semi-stable $\mbb{Q}_p$-representations of $G_K$
      with Hodge-Tate weights in $[0,r]$.
\item[--] The category of semi-stable $\mbb{Q}_p$-representations $V$ of $G_{K_n}$
      with the property that $V|_{G_{K_{\infty}}}$ is isomorphic to 
      $T_{\mfS_K}(\mfM)\otimes_{\mbb{Z}_p}\mbb{Q}_p$ 
      for some $\mfM\in  \mrm{Mod}^r_{/\mfS_K}$. 
\end{itemize}
\end{proposition}

The result below immediately follows from the above proposition.

\begin{corollary}
\label{Main3}
Let $T$ be a free $\mbb{Z}_p$-representation of $G_K$
which is semi-stable over $K_n$
for some $n\ge 0$. 
Then the following conditions are equivalent:
\begin{itemize}
\item[--]  $T|_{G_{K_{\infty}}}$ is isomorphic to 
      $T_{\mfS_K}(\mfM)$ for some $\mfM\in  \mrm{Mod}^r_{/\mfS_K}$. 
\item[--] There exists a semi-stable $\mbb{Q}_p$-representation $V$ of $G_K$
          with Hodge-Tate weights in $[0,r]$ such that 
          $T\otimes_{\mbb{Z}_p}\mbb{Q}_p$ is isomorphic to $V$ as representations 
          of $G_{K_{n'}}$ for some $n'\ge 0$.
\end{itemize}
\end{corollary}

\begin{remark}
\label{Rem:func}
In the statement of Corollary \ref{Main3},
we can always choose $n'$ to be $n$.
In addition, for a given $T$, $V$ is uniquely determined up to isomorphism.
Furthermore, the association $T\mapsto V$ is functorial. 
These follow from Proposition \ref{totst}.
\end{remark}

Combining this corollary with Theorem \ref{Thm:Liu} (3) and 
Remark \ref{Rem:Liu}, 
we obtain the desired relation $\mcal{C}^r \subset \mcal{C}^r_m$.
Therefore, it suffices to show Proposition \ref{Thm2}.
We begin with the following two lemmas.

\begin{lemma}
\label{exercise1}
For any $i\ge 0$,
we have a canonical decomposition
$$
\mrm{Fil}^iS_{K_n}=\bigoplus^{p^n-1}_{j=0}u^j_{K_n}\mrm{Fil}^iS_K. 
$$
\end{lemma}
\begin{proof}
Exercise.
\end{proof}

\begin{lemma}
\label{exercise2}
Let $\mfM$ be a Kisin module of height $\le r$ over $\mfS_K$.

\noindent
(1) $\mfM_n:=\mfS_{K_n}\otimes_{\mfS_K}\mfM$ is a Kisin module of height $\le r$ over $\mfS_{K_n}$
(with Frobenius $\vphi_{\mfM_n}:=\vphi_{\mfS_{K_n}}\otimes \vphi_{\mfM}$).

\noindent
(2) Let $\mcal{M}:=S_K\otimes_{\vphi,\mfS_K}\mfM$ and 
$\mcal{M}_n:=S_{K_n}\otimes_{\vphi,\mfS_K}\mfM=S_{K_n}\otimes_{\vphi,\mfS_{K_n}}\mfM_n$.
Define $\mrm{Fil}^i\mcal{M}
:=\{x\in \mcal{M} \mid (1\otimes \vphi_{\mfM})(x)\in \mrm{Fil}^iS_K\otimes_{\mfS_K} \mfM \}$
and 
$\mrm{Fil}^i\mcal{M}_n
:=\{x\in \mcal{M}_n \mid (1\otimes \vphi_{\mfM})(x)\in \mrm{Fil}^iS_{K_n}\otimes_{\mfS_{K}} \mfM\}
=\{x\in \mcal{M}_n \mid (1\otimes \vphi_{\mfM_n})(x)\in \mrm{Fil}^iS_{K_n}\otimes_{\mfS_{K_n}} \mfM_n \}$.
Then the natural isomorphism 
$S_{K_n}\otimes_{S_K}\mcal{M}\overset{\sim}{\rightarrow} \mcal{M}_n$
induces an isomorphism 
$S_{K_n}\otimes_{S_K}\mrm{Fil}^i\mcal{M}\overset{\sim}{\rightarrow} \mrm{Fil}^i\mcal{M}_n$.
\end{lemma}
\begin{proof}
The assertion (1) follows immediately 
by the relation $E_K(u_K)=E_{K_n}(u_{K_n})$.
In the rest of this proof we identify
$S_{K_n}\otimes_{S_K}\mcal{M}$ with $\mcal{M}_n$
by a natural way.
We show that 
$S_{K_n}\otimes_{S_K}\mrm{Fil}^i\mcal{M}=\mrm{Fil}^i\mcal{M}_n$.
The inclusion $S_{K_n}\otimes_{S_K}\mrm{Fil}^i\mcal{M}\subset \mrm{Fil}^i\mcal{M}_n$
follows from an easy calculation.
We have to prove the opposite inclusion.
Let $e_1,\dots ,e_d$ be an $\mfS_K$-basis of $\mfM$ and define a matrix $A\in M_d(\mfS_K)$
by
$\vphi_{\mfM}(e_1,\dots ,e_d)=(e_1,\dots ,e_d)A$.
We put $e^{\ast}_i=1\otimes e_i\in \vphi^{\ast}\mfM$ for any $i$.
Then $e^{\ast}_1,\dots ,e^{\ast}_d$ is an $S_{K_n}$-basis of $\mcal{M}_n$.
Take $x=\sum^d_{k=1}a_ke^{\ast}_k\in \mrm{Fil}^i\mcal{M}_n$ with $a_k\in S_{K_n}$.
Since $(1\otimes \vphi_{\mfM})(x)$ is contained in 
$\mrm{Fil}^iS_{K_n}\otimes_{\mfS} \mfM$,
we see that the matrix
$$
X:=A\begin{pmatrix}
a_1\\\rotatebox{90}{\dots}\\ a_d
\end{pmatrix}
$$
has coefficients in $\mrm{Fil}^iS_{K_n}$.
By Lemma \ref{exercise1},
each $a_k$ can be decomposed as 
$\sum^{p^n-1}_{j=0}u^j_{K_n}a^{(j)}_k$
for some $a^{(j)}_k\in S_K$.
Writing $A=(a_{lk})_{l,k}$ and $X={}^{\mrm{t}}(x_1,\dots ,x_d)$,
we have 
$$
x_l=\sum^d_{k=1}a_{lk}a_k=\sum^{p^n-1}_{j=0}u^j_{K_n}\sum^d_{k=1} a_{lk}a_k^{(j)}.
$$
By Lemma \ref{exercise1} again,
we obtain that $\sum^d_{k=1} a_{lk}a_k^{(j)}\in \mrm{Fil}^iS_K$.
If we put $x_{(j)}=\sum^d_{k=1}a_k^{(j)}e^{\ast}_k\in \mcal{M}$,
we have 
$$
(1\otimes \vphi_{\mfM})(x_{(j)})=\sum^d_{l=1}(\sum^d_{k=1} a_{lk}a_k^{(j)})e_l,
$$
which is contained in 
$\mrm{Fil}^iS_K\otimes_{\mfS_K}\mfM$.
Therefore, each $x_{(j)}$ is contained in $\mrm{Fil}^i\mcal{M}$.
Since $x=\sum^{p^n-1}_{j=0}u^j_{K_n}x_{(j)}$,
we obtain the fact that $x$ is contained in $S_{K_n}\otimes_{S_K}\mrm{Fil}^i\mcal{M}$.
\end{proof}

We proceed a proof of Proposition \ref{Thm2}. 
For simplicity,
we denote by $\mbf{R}_1$ (resp.\ $\mbf{R}_2$) 
the former (resp.\ latter) 
category appeared in the statement of Proposition \ref{Thm2}.
It is well-known  (cf. \cite[Lemma 2.1.15]{Ki}) that the essential image of $\mbf{R}_1$
under the restriction functor 
$\mrm{Rep}_{\mbb{Q}_p}(G_K)\to \mrm{Rep}_{\mbb{Q}_p}(G_{K_n})$
is contained in $\mbf{R}_2$.
Furthermore, 
the restriction functor $\mbf{R}_1\to \mbf{R}_2$ 
is fully faithful
since $K_n$ is totally ramified over $K$.
Thus it suffices to show the essential surjectivity
of the restriction functor $\mbf{R}_1\to \mbf{R}_2$.
Let $V$ be a semi-stable $\mbb{Q}_p$-representations $V$ of $G_{K_n}$
with the property that $V|_{G_{K_{\infty}}}$ is isomorphic to 
$T_{\mfS_K}(\mfM)\otimes_{\mbb{Z}_p}\mbb{Q}_p$ 
for some $\mfM\in  \mrm{Mod}^r_{/\mfS_K}$.
Set $T:=T_{\mfS_K}(\mfM)$ and 
take any $G_{K_n}$-stable $\mbb{Z}_p$-lattice $T'$ in $V$ such that $T\subset T'$. 
There exists a $(\vphi,\hat{G}_{K_n})$-module $\hat{\mfN}$ 
of height $\le r$ over $\mfS_{K_n}$
such that $T'\simeq \hat{T}_{K_n}(\hat{\mfN})$.
Put $\mfM_n=\mfS_{K_n}\otimes_{\mfS_K}\mfM$, 
which is a Kisin module of height $\le r$ over $\mfS_{K_n}$. 
Since the functor 
$T_{\mfS_{K_n}}$ from $\mrm{Mod}^r_{/\mfS_{K_n}}$ 
into $\mrm{Rep}_{\mbb{Z}_p}(G_{K_{\infty}})$
is fully faithful,
we obtain a morphism $\mfN\to \mfM_n$
which corresponds to the inclusion map $T\hookrightarrow T'$.
We note that it is injective and its cokernel $\mfM_n/\mfN$ 
is killed by a power of $p$
since $T'/T$ is $p$-power torsion.
Set 
$\mcal{D}_n
:=S_{K_n}[1/p]\otimes_{\mfS_{K_n}}\mfN \simeq S_{K_n}[1/p]\otimes_{\mfS_{K_n}}\mfM_n$,
$\mcal{D}:=S_K[1/p]\otimes_{\mfS_K}\mfM$,
$\mcal{N}:=S_{K_n}\otimes_{\mfS_{K_n}}\mfN$,
$\mcal{M}_n:=S_{K_n}\otimes_{\mfS_{K_n}}\mfM_n$ and 
$\mcal{M}:=S_K\otimes_{\mfS_K}\mfM$.
We define filtrations  
$\mrm{Fil}^i\mcal{N}$, 
$\mrm{Fil}^i\mcal{M}_n$ and 
$\mrm{Fil}^i\mcal{M}$
as Lemma \ref{exercise2} (2). 
Note that $\mcal{D}_n$
has a structure of a Breuil module which corresponds to $V$.
In particular, we have a Frobenius $\vphi_{\mcal{D}_n}$, a monodromy operator $N_{\mcal{D}_n}$
and a decreasing filtration $(\mrm{Fil}^i\mcal{D}_n)_{i\in \mbb{Z}}$ 
on $\mcal{D}_n$.
It is a result of \cite[\S 6]{Br} that we can equip
$D:=\mcal{D}_n/I_+S_{K_n}[1/p]\mcal{D}_n$
with a structure of filtered $(\vphi,N)$-module over $K_n$ which corresponds to $V$.
Now we recall the definition of this structure and also define some additional notations 
for later use.
The Frobenius $\vphi_D$ and the monodromy $N_D$ on $D$ is defined by
$\vphi_D:=\vphi_{\mcal{D}_n}\ \mrm{mod}\ I_+S_{K_n}[1/p]\mcal{D}_n$
and $N_D:=N_{\mcal{D}_n}\ \mrm{mod}\ I_+S_{K_n}[1/p]\mcal{D}_n$.
We denote by $f_{\pi_n}$ and $f_{\pi}$
the natural projections 
$\mcal{D}_n\twoheadrightarrow \mcal{D}_n/\mrm{Fil}^1S_{K_n}\mcal{D}_n$ 
and $\mcal{D}\twoheadrightarrow \mcal{D}/\mrm{Fil}^1S_K\mcal{D}$,
respectively.
There is a unique $\vphi$-compatible section $s\colon D\hookrightarrow \mcal{D}$
of the projection $\mcal{D}\twoheadrightarrow \mcal{D}/ I_+S_K[1/p]\mcal{D}\simeq D$.
Note that the composite $D\overset{s}{\hookrightarrow} \mcal{D}\hookrightarrow \mcal{D}_n$,
which is also denoted by $s$, is a section of the projection 
$\mcal{D}_n\twoheadrightarrow \mcal{D}_n/ I_+S_{K_n}[1/p]\mcal{D}_n=D$.
Since the composite 
$D\overset{s}{\to} \mcal{D}_n\overset{f_{\pi_n}}{\to} \mcal{D}_n/\mrm{Fil}^1S_{K_n}\mcal{D}_n$
(resp.\ $D\overset{s}{\to} \mcal{D}\overset{f_{\pi}}{\to} \mcal{D}/\mrm{Fil}^1S_K\mcal{D}$)
maps a basis of $D$ to a basis of 
$\mcal{D}_n/\mrm{Fil}^1S_{K_n}\mcal{D}_n$
(resp.\ $\mcal{D}/\mrm{Fil}^1S_K\mcal{D}$),
we obtain an isomorphism 
$D_{K_n}:=K_n\otimes_{W(k)[1/p]} D\overset{\sim}{\rightarrow} 
\mcal{D}_n/\mrm{Fil}^1S_{K_n}\mcal{D}_n$
(resp.\ $D_K:=K\otimes_{W(k)[1/p]} D\overset{\sim}{\rightarrow} 
\mcal{D}/\mrm{Fil}^1S_K\mcal{D}$).
By this isomorphism,
we identify $D_{K_n}$ (resp.\ $D_K$)
with $\mcal{D}_n/\mrm{Fil}^1S_{K_n}\mcal{D}_n$ (resp.\ $\mcal{D}/\mrm{Fil}^1S_K\mcal{D}$).
Then the filtration $(\mrm{Fil}^iD_{K_n})_{i\in \mbb{Z}}$ on 
$D$ over $K_n$ is given by
$\mrm{Fil}^iD_{K_n}=f_{\pi_n}(\mrm{Fil}^i\mcal{D}_n)$. 
We note that the filtered $(\vphi,N)$-module $D$ over $K_n$ defined above
is weakly admissible since $V$ is semi-stable
(see \cite[\S 3.4]{CF} for the definition of weakly admissibility).

Let $\cO_{K}$ and $\cO_{K_n}$ be rings of integers of $K$ and $K_n$, respectively.
We note that there exists a canonical isomorphism
$K_n\otimes_{\cO_{K_n}} f_{\pi_n}(\mrm{Fil}^i\mcal{M}_n)
\simeq K_n\otimes_{\cO_{K_n}} f_{\pi_n}(\mrm{Fil}^i\mcal{N})$ since
we have $p^c\mrm{Fil}^i\mcal{M}_n\subset \mrm{Fil}^i\mcal{N}\subset \mrm{Fil}^i\mcal{M}_n$
as submodules of $\mcal{D}_n$,
where $c\ge 0$ is an integer such that 
$\mfM_n/\mfN$ is killed by $p^c$.
On the other hand,
the canonical isomorphism 
$S_{K_n}[1/p]\otimes_{S_K[1/p]} \mcal{D}
\simeq  \mcal{D}_n$ induces an isomorphism
$S_{K_n}\otimes_{S_K} \mrm{Fil}^i\mcal{M}
\simeq  \mrm{Fil}^i\mcal{M}_n$
(cf.\ Lemma \ref{exercise2} (2)), 
and it gives an isomorphism
$\cO_{K_n}\otimes_{\cO_K}f_{\pi}(\mrm{Fil}^i\mcal{M})\simeq 
f_{\pi_n}(\mrm{Fil}^i\mcal{M}_n)$.
Furthermore, it follows from 
\cite[Corollary 3.2.3]{Li2} 
that a natural isomorphism 
$\mcal{N}[1/p]\simeq \mcal{D}_n$
preserves filtrations, where $\mrm{Fil}^i(\mcal{N}[1/p]):=
(\mrm{Fil}^i\mcal{N})[1/p]$.
This induces
$K_n\otimes_{\cO_{K_n}} f_{\pi_n}(\mrm{Fil}^i\mcal{N})\simeq 
\mrm{Fil}^iD_{K_n}$.
(Here, we remark that the argument of \S 3.2 of {\it loc.\ cit.} proceeds even for $p=2$.)
Therefore, if we define a decreasing filtration $(\mrm{Fil}^iD_K)_{i\in \mbb{Z}}$ on $D_K$ by
$\mrm{Fil}^iD_K:=K\otimes_{\cO_K} f_{\pi}(\mrm{Fil}^i\mcal{M})$,
then we have a canonical isomorphism
\begin{equation}
\label{canonical}
K_n\otimes_K \mrm{Fil}^iD_K
\simeq \mrm{Fil}^iD_{K_n}.
\end{equation}
Note that 
we know $D_K=\mrm{Fil}^0D_K\supset \mrm{Fil}^1D_K
\supset \cdots \supset \mrm{Fil}^{r+1}D_K=0$.
Now we recall that 
$D$ is weakly admissible as a filtered $(\vphi,N)$-module over $K_n$.
It follows from (\ref{canonical})
that $D$ is also weakly admissible
as a filtered $(\vphi,N)$-module over $K$,
and hence
the action of  $G_{K_n}$ on $V$ extends to $G_K$ 
so that it is semi-stable over $K$.
Therefore, we showed that the restriction functor $\mbf{R}_1\to \mbf{R}_2$ is essentially surjective and 
this finishes a proof of Proposition \ref{Thm2}.

%%%%%%%%%%%%%%%%%%%%%%%%%%%%%%%%%%%%%%%%%%%%%%%%%%%%%%%%%%%%%%%%%%%%%%%%%%%%%%%%%%%%%%%%%%%%%%%%%%%%%%%%%%%
%%%%%%%%%%%%%%%%%%%%%%%%%%%%%%%%%%%%%%%%%%%%%%%%%%%%%%%%%%%%%%%%%%%%%%%%%%%%%%%%%%%%%%%%%%%%%%%%%%%%%%%%%%%
%                           3.4                              %%%%%%%%%%%%%%%%%%%%%%%%%%%%%%%%%%%%%%%%%%%%%%%%
%%%%%%%%%%%%%%%%%%%%%%%%%%%%%%%%%%%%%%%%%%%%%%%%%%%%%%%%%%%%%%%%%%%%%%%%%%%%%%%%%%%%%%%%%%%%%%%%%%%%%%%%%%%
%%%%%%%%%%%%%%%%%%%%%%%%%%%%%%%%%%%%%%%%%%%%%%%%%%%%%%%%%%%%%%%%%%%%%%%%%%%%%%%%%%%%%%%%%%%%%%%%%%%%%%%%%%%

\subsection{$\mcal{C}^r_{m_0}= \mcal{C}^r$}
\label{3.4}

Now we are ready to complete a proof of Theorem \ref{Main1'}.
We put $K_{p^{\infty}}=\bigcup_{i\ge 0} K(\zeta_{p^i})$ and 
$G_{p^{\infty}}=\mrm{Gal}(K_{\infty}K_{p^{\infty}}/K_{p^{\infty}})\subset \hat{G}_K$.
We fix a topological generator $\tau$ of
$G_{p^{\infty}}$.
We start with the following lemma.

\begin{lemma}
(1) The field $K_{p^{\infty}}\cap K_{\infty}$ coincides with $K$ or $K_1$.

\noindent
(2) If $(p,m_0)\not=(2,1)$, then $K_{p^{\infty}}\cap K_{\infty}=K$.

\noindent
(3) If $m\ge 2$, then $K_{p^{\infty}}\cap K_{\infty}=K$.
\end{lemma}
\begin{proof}
The assertions (1) and (2)
are consequences of \cite[Lemma 5.1.2]{Li2} and  \cite[Proposition 4.1.5]{Li3},
and so it is enough to show (3). We may assume $p=2$.
Assume that $K_{p^{\infty}}\cap K_{\infty}\not=K$.
Then we have $K_{p^{\infty}}\cap K_{\infty}=K_1$ by (1).
Since $K_1$ is contained in $K_{p^{\infty}}$,
we have $K_1\subset K(\zeta_{2^{\ell}})$ for $\ell>m$ large enough.
Since $m\ge 2$, the extension $K(\zeta_{2^{\ell}})/K(\zeta_{2^m})$
is cyclic and thus there exists only one quadratic subextension in it.  
By definition of $m$, the extension $K(\zeta_{2^{m+1}})/K(\zeta_{2^m})$
is degree $2$. 
Since the extension $K_1/K$ is totally ramified but $K(\zeta_{2^m})/K$ is unramified,
we see that the extension $K_1(\zeta_{2^m})/K(\zeta_{2^m})$
is also degree $2$.
Therefore, we have
$K_1(\zeta_{2^m})=K(\zeta_{2^{m+1}})$, and then
we have $\pi_1=x\zeta_{2^{m+1}}+y$ with $x,y\in K(\zeta_{2^m})$.
Let $\sigma$
be a non-trivial element in $\mrm{Gal}(K(\zeta_{2^{m+1}})/K(\zeta_{2^m}))$.
We have 
$-\pi_1=\sigma(\pi_1)=x\sigma(\zeta_{2^{m+1}})+y=-x\zeta_{2^{m+1}}+y$.
Hence $\pi_1=x\zeta_{2^{m+1}}$ and we have $v(\pi_1)=v(x)$.
Here, $v$ is a valuation of $K(\zeta_{2^{m+1}})$
normalized by $v(K^{\times})=\mbb{Z}$, and we see $v(\pi_1)=1/2$.
Since the extension $K(\zeta_{2^m})/K$ is unramified,
we have $v(x)\in \mbb{Z}$ but this is a contradiction.
\end{proof}

If $(p,m_0)=(2,1)$ and $m=1$, 
we have $m_0=m$ and then Theorem \ref{Main1'} follows immediately
from Lemma \ref{Lem:Main1'}. 
Hence we may assume $(p,m_0)\not=(2,1)$ or $m\ge 2$.
Under this assumption, the above lemma implies 
$K_{p^{\infty}}\cap K_{\infty}=K$.
In particular,  we have $\hat{G}=G_{p^{\infty}} \rtimes H_K$
with the relation $g\sigma=\sigma^{\chi(g)}g$
for $g\in H_K$ and $\sigma\in G_{p^{\infty}}$.
Here, $\chi$ is the $p$-adic cyclotomic character.
Let $\hat{\mfM}=(\mfM,\vphi,\hat{G}_K)$ be an object of 
${}_{\mrm{w}}\Mod^{r,\hat{G}_K}_{/\mfS_K}$ and put $T=\hat{T}_K(\hat{\mfM})$.
Our goal is to show that $T$ is an object of $\mcal{C}^r_{m_0}$.
We put $\mcal{D}=S_{K}[1/p]\otimes_{\vphi,\mfS_K} \mfM$
and $D=\mcal{D}/I_+S_K[1/p]\mcal{D}$. 
Let $s\colon D\hookrightarrow \mcal{D}$ be a $\vphi$-equivariant 
$W(k)[1/p]$-linear section of the projection $\mcal{D}\twoheadrightarrow D$ 
as before, and take a basis $e_1,\dots ,e_d$ of $s(D)$.
In $\mcal{R}_K\otimes_{W(k)[1/p]} s(D)=\mcal{R}_K\otimes_{\vphi,\mfS_K} \mfM$,  
the $\tau$-action with respected to the basis $e_1,\dots ,e_d$ is given by 
$\tau(e_1,\dots e_d)=(e_1,\dots ,e_d)A(t)$
for some matrix $A(t)\in GL_d(W(k)[1/p][\![t]\!])$.
Moreover, we have 
$\hat{G}_K(s(D))\subset (\mcal{R}_K\cap W(k)[1/p][\![t]\!])\otimes_{W(k)[1/p]} s(D)$
by \cite[Lemma 7.1.3]{Li1}.
Here are two remarks.
The first one is that, the $a$-th power $A(t)^a$, a matrix with coefficients 
in $W(k)[1/p][\![t]\!]$, of $A(t)$ 
is well-defined for any $a\in \mbb{Z}_p$. 
This is because the Galois group
$G_{p^{\infty}}=\tau^{\mbb{Z}_p}
\subset \hat{G}_K$ acts continuously on  $\mcal{R}_K\otimes_{W(k)[1/p]} s(D)$.
The second one is that, for any $g\in H_K$, we have $A(\chi(g)t)=A(t)^{\chi(g)}$
by the relation $g\tau=\tau^{\chi(g)}g$. In particular,  we have 
\begin{equation}
\label{Keyrel}
A(0)^{\chi(g)-1}=I_d.
\end{equation}
\noindent
Here, $I_d$ is the identity matrix.
With these notation,
it follows from the second paragraph of the proof of  \cite[Theorem 4.2.2]{Li3} that
$T\otimes_{\mbb{Z}_p} \mbb{Q}_p$ is semi-stable over $K_{\ell}$
if $A(0)^{p^{\ell}}=I_d$.

\begin{lemma}
\label{lastlemma}
Let the notation be as above.
Then we have $A(0)^{p^{m_0}}=I_d$. 
\end{lemma}

\begin{proof}
First we consider the case where  $p$ is odd.
Since $H_K$ is canonically isomorphic to $\mrm{Gal}(K_{p^{\infty}}/K)$,
the image of the restriction to $H_K$ of
the $p$-adic cyclotomic character 
$\chi\colon \hat{G}_K\to \mbb{Z}_p^{\times}$
 is equal to  
$$
\chi(\hat{G}_K)=C\times (1+p^n\mbb{Z}_p) 
$$
\noindent
where $n$ is a positive integer and  
$C\simeq \mrm{Gal}(K(\zeta_p)/K)$ is a finite cyclic group of order prime-to-$p$.

\noindent
{\it The case where $m_0\ge 1$:} 
In this case, it is an easy exercise to check the equality $n=m_0$ and hence
we can choose $g\in H_K$ such that $\chi(g)=1+p^{m_0}$.
Thus the result follows by (\ref{Keyrel}).

\noindent
{\it The case where $m_0=0$:} 
In this case, $C$ is non-trivial and hence 
there exists an element $g\in H_K$
such that $x:=\chi(g)-1$ is a unit of $\mbb{Z}_p$.
By (\ref{Keyrel}), we have $A(0)^x=I_d$, and then 
we obtain $A(0)=I_d$.

Next we consider the case where $p=2$.

\noindent
{\it The case where $m_0\ge 2$:} 
This case is clear
since we have $\chi(H_K)=\chi(\hat{G}_K)=1+2^{m_0}\mbb{Z}_2$.

\noindent
{\it The case where $m_0=1$:} 
In this case, $\chi\ \mrm{mod}\ 4$ is not trivial.
Hence there exists 
$g\in H_K$ such that $\chi(g)=3+4x$ for some $x\in \mbb{Z}_2$.
By (\ref{Keyrel}),
we have $A(0)^{2+4x}=I_d$. 
Since $1+2x$ is a unit of $\mbb{Z}_2$,
this gives the desired equation $A(0)^2=I_d$.
\end{proof}

\noindent
By the above lemma, we obtain the fact that 
$T\otimes_{\mbb{Z}_p}\mbb{Q}_p$ is semi-stable over $K_{m_0}$.
On the other hand, we have already shown 
that $\mcal{C}^r$ is a subcategory of $\mcal{C}^r_{m}$.
Thus there exists a semi-stable $\mbb{Q}_p$-representation $V$ of $G_K$
whose restriction to $G_{K_m}$ is isomorphic to $T\otimes_{\mbb{Z}_p} \mbb{Q}_p$. 
Moreover, Proposition \ref{totst} implies that $V$ and  $T\otimes_{\mbb{Z}_p} \mbb{Q}_p$ are isomorphic 
as representations of $G_{K_{m_0}}$
since they are semi-stable over $K_{m_0}$.
Therefore, we conclude that $T$ is an object of the category  $\mcal{C}^r_{m_0}$.
This is the end of  a proof of Theorem \ref{Main1'}. 

%%%%%%%%%%%%%%%%%%%%%%%%%%%%%%%%%%%%%%%%%%%%%%%%%%%%%%%%%%%%%%%%%%%%%%%%%%%%%%%%%%%%%%%%%%%%%%%%%%%%%%%%%%%
%%%%%%%%%%%%%%%%%%%%%%%%%%%%%%%%%%%%%%%%%%%%%%%%%%%%%%%%%%%%%%%%%%%%%%%%%%%%%%%%%%%%%%%%%%%%%%%%%%%%%%%%%%%
%                           3.5                              %%%%%%%%%%%%%%%%%%%%%%%%%%%%%%%%%%%%%%%%%%%%%%%%
%%%%%%%%%%%%%%%%%%%%%%%%%%%%%%%%%%%%%%%%%%%%%%%%%%%%%%%%%%%%%%%%%%%%%%%%%%%%%%%%%%%%%%%%%%%%%%%%%%%%%%%%%%%
%%%%%%%%%%%%%%%%%%%%%%%%%%%%%%%%%%%%%%%%%%%%%%%%%%%%%%%%%%%%%%%%%%%%%%%%%%%%%%%%%%%%%%%%%%%%%%%%%%%%%%%%%%%

\subsection{Conclusions and more}

\subsubsection{}
We summarize our results here.
For  any finite extension $L/K$, 
we denote by $\mrm{Rep}^{r,L\mathchar`-\mrm{st}}_{\mbb{Z}_p}(G_K)$ the category
of free $\mbb{Z}_p$-representations $T$ of $G_K$
which is semi-stable over $L$ with Hodge-Tate weights in $[0,r]$.
We define $\mcal{C}^r_n$ to be the category of 
free $\mbb{Z}_p$-representations $T$ of $G_K$ which satisfies the following property:
there exists a semi-stable $\mbb{Q}_p$-representation $V$ of $G_K$
with Hodge-Tate weights in $[0,r]$
such that $T\otimes_{\mbb{Z}_p}\mbb{Q}_p$
is isomorphic to $V$ as representations of $G_{K_n}$. 
By definition $\mcal{C}^r_n$ is a full subcategory of  
$\mrm{Rep}^{r,K_n\mathchar`-\mrm{st}}_{\mbb{Z}_p}(G_K)$.
Put $m_0=\mrm{max}\{i\ge 0 \mid \zeta_{p^i}\in K \}$ and 
$m=\mrm{max}\{i\ge 0 \mid \zeta_{p^i}\in K^{\mrm{ur}} \}$.
We have  
$\mrm{Rep}^{r,K_m\mathchar`-\mrm{st}}_{\mbb{Z}_p}(G_K)
=\bigcup_{n\ge 0} \mrm{Rep}^{r,K_n\mathchar`-\mrm{st}}_{\mbb{Z}_p}(G_K)$,
$\mcal{C}^r_m=\bigcup_{n\ge 0} \mcal{C}^r_n$
(see Remark \ref{Rem:Liu}).
Results of \cite{Li3} and this note give the following diagram 
(here, ``$\subset$''  implies an inclusion):
\begin{center}
$\displaystyle \xymatrix{
& 
&  \mcal{C}^r_m   \ar^{\subset \quad \qquad}[r]
& \mrm{Rep}^{r,K_m\mathchar`-\mrm{st}}_{\mbb{Z}_p}(G_K)
\\ 
{}_{\mrm{w}}\Mod^{r,\hat{G}_K}_{/\mfS_K} 
\ar^{\sim}@{->}_{\hat{T}_K}[rr] 
&
& \mcal{C}^r_{m_0} \ar^{\cup}[u].  \ar^{\subset \quad \qquad}[r]
& \mrm{Rep}^{r,K_{m_0}\mathchar`-\mrm{st}}_{\mbb{Z}_p}(G_K)  \ar^{\cup}[u]
\\
\Mod^{r,\hat{G}_K}_{/\mfS_K}  
\ar^{\cup}[u] \ar^{\sim}_{\hat{T}_K}[rr] 
&
&
\mcal{C}^r_{0} \ar^{\cup}[u].  \ar@{=}[r]
&
\mrm{Rep}^{r,K\mathchar`-\mrm{st}}_{\mbb{Z}_p}(G_K) \ar^{\cup}[u]
}$
\end{center}

\subsubsection{}
We give a few remarks for the above diagram.
Clearly, all the categories in the middle and right vertical lines are same
if $m=0$.
On the other hand, if $m \ge 1$,
inclusion relations between them are described as follows:

\begin{proposition}
\label{prop:rem}
Suppose $m\ge 1$.

\noindent
(1) Suppose $1\le n\le m$.
Then the category  $\mcal{C}^r_n$ is strictly larger than 
$\mcal{C}^r_{n-1}$.
In particular, the category
$\mrm{Rep}^{r,K_{n}\mathchar`-\mrm{st}}_{\mbb{Z}_p}(G_K)$ 
is strictly larger than 
$\mrm{Rep}^{r,K_{n-1}\mathchar`-\mrm{st}}_{\mbb{Z}_p}(G_K)$.

\noindent
(2) Suppose $n,r\ge 1$.
Then the category
$\mrm{Rep}^{r,K_n\mathchar`-\mrm{st}}_{\mbb{Z}_p}(G_K)$ 
is strictly larger than 
$\mcal{C}^r_n$.

\noindent
(3) Suppose $n\ge 0$. 
Then we have 
$\mcal{C}^0_n=\mrm{Rep}^{0,K_n\mathchar`-\mrm{st}}_{\mbb{Z}_p}(G_K)$.
\end{proposition}

\begin{proof}
(1) Let $T$ be the induced representation of the rank one trivial $\mbb{Z}_p$-representation
of $G_{K_n(\zeta_{p^n})}$ to $G_K$,
which is an Artin representation.
The splitting field of $T$ is $K_n(\zeta_{p^n})$.
Since $n\le m$,  the extension 
$K_n(\zeta_{p^n})/K_n$ is unramified. 
Thus $T$ is crystalline over $K_n$.
On the other hand,
$T$ is not crystalline over $K_{n-1}$ since the extension 
$K_n(\zeta_{p^n})/K_{n-1}$, the splitting field of $T|_{K_{n-1}}$, is not unramified.
(This finishes a proof of the latter assertion.)
Let $\rho_T\colon G_K\to GL_{\mbb{Z}_p}(T)\simeq GL_d(\mbb{Z}_p)$
be the continuous homomorphism associated with $T$, 
where $d$ is the $\mbb{Z}_p$-rank of $T$.
By the assumption $n\le m$,
we know that $K(\zeta_{p^n})\cap K_n=K$ and thus we can define 
a continuous homomorphism $\rho_{T'}\colon G_K\to GL_d(\mbb{Z}_p)$
by the composite 
$G_K\twoheadrightarrow \mrm{Gal}(K(\zeta_{p^n})/K)\simeq \mrm{Gal}(K_n(\zeta_{p^n})/K_n)
\overset{\rho_T}{\hookrightarrow}  GL_d(\mbb{Z}_p)$.
Let $T'$ be the free $\mbb{Z}_p$-module of rank $d$ equipped with a 
$G_K$-action by $\rho_{T'}$.
Then $T'$ is isomorphic to $T$ as representations of $G_{K_n}$
and furthermore it is crystalline over $K$.
It follows that  $T$ is an object of $\mcal{C}^r_n$.

\noindent
(2) Since $m\ge 1$,  we know that $L:=K(\zeta_p)$ is an unramified extension of 
$K$. Thus $\pi_K$ is a uniformizer of $L$.
Consider notations 
$\mfS_L, S_L,\dots $ (resp.\ $\mfS_{L_1}, S_{L_1},\dots$)
with respect to the uniformizer $\pi_K$ (resp.\ $\pi_{K,1}$) 
of $L$ (resp.\ $L_1$)
and the system $(\pi_{K,n})_{n\ge 0}$ (resp.\ $(\pi_{K,n+1})_{n\ge 0}$).
Let $\mfM$ be the rank-$2$ free Kisin module over $\mfS_{L_1}$ 
of height $1$ given by 
$\vphi(e_1,e_2)=(e_1,e_2)\begin{pmatrix}1 & u_{L_1} \\ 0 & E_{L_1}(u_{L_1})\end{pmatrix}$,
where $\{e_1,e_2\}$ is a basis of $\mfM$. 
Since $\mfM$ is of height $1$, there exists a
$G_{L_1}$-stable $\mbb{Z}_p$-lattice $T$ in a 
crystalline $\mbb{Q}_p$-representation of $G_{L_1}$, coming 
from a $p$-divisible group over the integer ring of $L_1$.
We see that $\tilde{T}:=\mrm{Ind}^{G_K}_{G_{L_1}}T$ is crystalline over $L_1$.
Since $L_1$ is unramified over $K_1$,
$\tilde{T}$ is in fact crystalline over $K_1$.
Furthermore, $\tilde{T}$  does not come from Kisin modules over $\mfS_K$
(that is, $\tilde{T}|_{G_{K_{\infty}}}$ is not isomorphic to $T_{\mfS_K}(\mfN)$ for 
any Kisin module $\mfN$ over $\mfS_K$).
To check this, it suffices to show that $\tilde{T}$ does not 
come from Kisin modules over $\mfS_L$. 
Essentially, this has been already shown in \cite[Example 4.2.3]{Li3}.
Therefore, Corollary \ref{Main3} implies that $\tilde{T}$
is not an object of $\mcal{C}^r_n$.

\noindent
(3) We may suppose $n\le m$.
Take any object $T$ of $\mrm{Rep}^{0,K_n\mathchar`-\mrm{st}}_{\mbb{Z}_p}(G_K)$.
Since $T$ has only one Hodge-Tate weight zero,
the condition $T|_{G_{K_n}}$ is semi-stable  implies  that $T|_{G_{K_n}}$ is unramified.
Thus if we denote by $K_T$ the splitting field of $T$,
then  $K_TK_n$
is unramified over $K_n$.

First we consider the case where $K_T$ contains $\zeta_{p^n}$.
In this case, we follow the idea given in the proof of (1).
Denote by $K'$ the maximum unramified subextension 
of $K_TK_n$ over $K$. 
Since $K_T$ contains $\zeta_{p^n}$,  
$K_TK_n/K$ is a Galois extension and 
hence $K'/K$ is also Galois.
Furthermore, it is not difficult to check that the equality  $K_TK_n=K'K_n$ holds.
Let $\rho_T\colon G_K\to GL_{\mbb{Z}_p}(T)\simeq GL_d(\mbb{Z}_p)$
be the continuous homomorphism associated with $T$, 
where $d$ is the $\mbb{Z}_p$-rank of $T$,
and define a continuous homomorphism $\rho_{T'}\colon G_K\to GL_d(\mbb{Z}_p)$
by the composite 
$G_K\twoheadrightarrow \mrm{Gal}(K'/K)\simeq \mrm{Gal}(K_TK_n/K_n)
\overset{\rho_T}{\hookrightarrow}  GL_d(\mbb{Z}_p)$.
Let $T'$ be the free $\mbb{Z}_p$-module of rank $d$ equipped with a 
$G_K$-action by $\rho_{T'}$.
Then $T'$ is isomorphic to $T$ as representations of $G_{K_n}$
and furthermore, $T'$ is crystalline over $K$.
It follows that  $T$ is an object of $\mcal{C}^0_n$.

Next we consider the general case. 
Denote by $T_0$ the induced representation of the rank one trivial $\mbb{Z}_p$-representation
of $G_{K(\zeta_{p^n})}$ to $G_K$.
We define a free $\mbb{Z}_p$-representation $\tilde{T}$ of $G_K$ 
by $\tilde{T}:=T\oplus T_0$.
The splitting fields of  $\tilde{T}$ and  $T_0$ are  
equal to $K_{\tilde{T}}:=K_T(\zeta_{p^n})$ and $K(\zeta_{p^n})$,
respectively.
The representations $\tilde{T}$ and $T_0$ are objects of 
$\mrm{Rep}^{0,K_n\mathchar`-\mrm{st}}_{\mbb{Z}_p}(G_K)$.
Moreover, the above argument implies that 
$\tilde{T}$ and $T_0$ are contained in $\mcal{C}^0_n$.
Therefore, there exist objects $\tilde{\mfM}$ and $\mfM_0$ of 
$\mrm{Mod}^r_{/\mfS_K}$ such that 
$T_{\mfS_K}(\tilde{\mfM})=\tilde{T}|_{G_{K_{\infty}}}$ and $T_{\mfS_K}(\mfM_0)=T_0|_{G_{K_{\infty}}}$.
Now we recall that the functor $T_{\mfS_K}$ is fully faithful. 
If we denote by  
$\mfrak{f}\colon \mfM_0\to \tilde{\mfM}$ a (unique) morphism of $\vphi$-modules over $\mfS_K$
corresponding to the natural projection $\tilde{T}\twoheadrightarrow T_0$,
then we obtain a split exact sequence 
$0\to \mfM_0\overset{\mfrak{f}}{\to} \tilde{\mfM}\to \mfM\to 0$
of $\vphi$-modules over $\mfS_K$.
Here, $\mfM$ is the cokernel of $\mfrak{f}$, which is a finitely generated  $\mfS_K$-module.
Since $\mfM$ is a direct summand of $\tilde{\mfM}$,
it is a projective $\mfS_K$-module. This implies that $\mfM$ is a free $\mfS_K$-module.
(Note that, for a finitely generated $\mfS_K$-module, 
it is projective over $\mfS_K$ if and only if  it is free $\mfS_K$ by Nakayama's lemma.)
Furthermore, $\mfM$ is of height $0$ and hence it is an object of  $\mrm{Mod}^0_{/\mfS_K}$.
Since the functor $T_{\mfS_K}$ is exact, 
we obtain 
$T_{\mfS_K}(\mfM)
=\mrm{ker}(T_{\mfS_K}(\tilde{\mfM})\overset{T_{\mfS_K}(\mfrak{f})}{\longrightarrow}  T_{\mfS_K}(\mfM_0))
=\mrm{ker}(\tilde{T}\twoheadrightarrow T_0)=T$.
Therefore, $T$ is an object of  $\mcal{C}^0_n$
by Corollary \ref{Main3}. 
\end{proof}

\subsubsection{}
Assume that $m\ge 1$.
Let $n\ge 1$ be an integer and  $T$ an object of the category $\mcal{C}^r_n$.
By definition of $\mcal{C}^r_n$, we have a (unique) semi-stable $\mbb{Q}_p$-representation 
$V_T$ of $G_K$
with the property that it is isomorphic to $T\otimes_{\mbb{Z}_p} \mbb{Q}_p$ 
as representations of $G_{K_n}$.
It is not clear  whether
$T$ is stable under the $G_K$-action of $V_T$ for any $T$ or not.
Such a stability problem of Galois actions may sometimes cause obstructions in integral theory,
and so the following question should be naturally considered.

\begin{question}
\label{question2}
Let the notation be as above. 
Does the $G_K$-action of $V_T$ preserves $T$ for any $T$?
\end{question}

%In the rest of this paper, we describe an example which gives a
%negative answer to this question for some cases.

We end this paper by showing an answer to this question.

\begin{proposition} 
(1) If $r=0$, then Question \ref{question2}
has an affirmative answer.

\noindent
(2) If $r\ge 1$, then Question \ref{question2}
has a negative answer.

\noindent
(3) Let the notation be as above. Suppose $e(r-1)<p-1$
where $e$ is the absolute ramification index of $K$.
If $T$ is potentially crystalline, then the $G_K$-action of $V_T$ preserves $T$.
Moreover, any $G_{K_{\infty}}$-stable $\mbb{Z}_p$-lattice of $V_T$
is stable under the $G_K$-action.
\end{proposition}
%\begin{example}

\begin{proof}
(1) (This is a special case of (3).) 
The result easily follows from the fact that
$T$ as in the question is unramified in this case, and 
that $G_{K_n}$ and the inertia subgroup of $G_K$ generate $G_K$.

\noindent
(2) Our goal is to construct an example which gives a negative answer to the question.
First we consider the case where 
$1\le n\le m_0$.
Let $E_{\pi}$ be the Tate curve over $K$ associated to $\pi$.
Choose a basis $\{\mbf{e},\ \mbf{f}\}$ of the  $p$-adic Tate module $V=V_p(E_{\pi})$
 of $E_{\pi}$ such that the $G_K$-action on $V$ with respective to this basis is given by 
$$
g\mapsto 
\begin{pmatrix}
\chi(g) & c(g)\\
0 & 1
\end{pmatrix}.
$$ 
Here, $\chi\colon G_K\to \mbb{Z}^{\times}_p$ is the $p$-adic cyclotomic character
and $c\colon G_K\to \mbb{Z}_p$ is a map defined by 
$g(\pi_{K,{\ell}})=\zeta^{c(g)}_{p^{\ell}}\pi_{K,{\ell}}$ for any $g\in G_K$ and $\ell\ge 1$.
Let $T_0$ be the free $\mbb{Z}_p$-submodule of $V$ generated by 
$p^n\mbf{e}$ and $\mbf{f}$.
This is $G_{K_n}$-stable but not $G_K$-stable in $V$.
Now we put $T=\mrm{Ind}^{G_K}_{G_{K_n}}T_0$ and choose a
set $S\subset G_K$  of representatives of the quotient $G_K/G_{K_n}$.
Since $K_n/K$ is Galois, $T|_{G_{K_n}}$ is of the form $\oplus_{\sigma\in S}\ T_{0,\sigma}$.
Here, $T_{0,\sigma}$ is just $T_0$ as a $\mbb{Z}_p$-module and is
equipped with a $\sigma$-twisted $G_{K_n}$-action, that is, $g.x:=(\sigma^{-1} g \sigma)(x)$
for $g\in G_{K_n}$ and $x\in T_{0,\sigma}$.
We define elements $\mbf{e}_{\sigma}$ and $\mbf{f}_{\sigma}$ 
of $T_{0,\sigma}$
by  $\mbf{e}_{\sigma}:=p^n\mbf{e}$ and 
$\mbf{f}_{\sigma}:=\mbf{f}$.
We define $V_{0,\sigma}:=T_{0,\sigma}\otimes_{\mbb{Z}_p} \mbb{Q}_p$
and extend the $G_{K_n}$-action on $V_{0,\sigma}$ to 
$G_K$ by 
$$
g(\mbf{e}_{\sigma},\mbf{f}_{\sigma})=(\mbf{e}_{\sigma},\mbf{f}_{\sigma}) 
\begin{pmatrix}
\chi(g) & c(\sigma^{-1}g\sigma)/p^n\\
0 & 1
\end{pmatrix}
$$ 
for $g\in G_K$.
By definition the $G_K$-action on $V_{0,\sigma}$
does not preserve $T_{0,\sigma}$. It is not difficult to check that 
$V_{0,\sigma}$ is a semi-stable $\mbb{Q}_p$-representation  of $G_K$
with Hodge-Tate weights $\{0,1\}$.
If we put $V_T=\oplus_{\sigma\in S} V_{0,\sigma}$,
then we have the followings:
\begin{itemize}
\item $V_T$ is semi-stable over $K$ with Hodge-Tate weights $\{0,1\}$,
\item the natural isomorphism $V_T\simeq T\otimes_{\mbb{Z}_p} \mbb{Q}_p$ 
is compatible with $G_{K_n}$-actions, and 
\item the $G_K$-action on $V_T$ does not preserve $T$. 
\end{itemize} 
This gives a negative answer to Question \ref{question}
in the case $1\le n\le m_0$.

Next we consider a general case. We may suppose $n=m$.
Put $K'=K(\zeta_{p^m})$ and $K'_m=K_mK'$. 
Then $K'$ is an unramified Galois extension of $K$
and $\mrm{max}\{i\ge 0 \mid \zeta_{p^i}\in K' \}=m$.
Thus the above argument shows that
there exists a free $\mbb{Z}_p$-representation $T'$
of $G_{K'}$ and a semi-stable $\mbb{Q}_p$-representation $V_{T'}$ of $G_{K'}$
with Hodge-Tate weights $\{0,1\}$
which satisfies the followings:
\begin{itemize}
\item there exists an isomorphism  
$V_{T'}\simeq T'\otimes_{\mbb{Z}_p} \mbb{Q}_p$
of $G_{K'_m}$-representations, and 
\item the $G_{K'}$-action on $V_{T'}$ does not preserve $T'$. 
\end{itemize} 
We regard $T'$ as a $\mbb{Z}_p$-lattice of $V_{T'}$.
We define $T:=\mrm{Ind}^{G_K}_{G_{K'}}T'$ and $V_T:=\mrm{Ind}^{G_K}_{G_{K'}}V_{T'}$.
Note that $T$ is naturally regarded as a $\mbb{Z}_p$-lattice of $V$.
By definition, the $G_{K'}$-action  on $V_T$ does not preserve $T$.
In particular, the same holds also for the $G_K$-action. 
Since $K'/K$ is unramified, we see that $V_T$ 
is semi-stable over $K$.
Furthermore,
by Mackey's formula, we have natural  isomorphisms 
$T\otimes_{\mbb{Z}_p}\mbb{Q}_p
\simeq \mrm{Ind}^{G_{K_m}}_{G_{K'_m}}(T'\otimes_{\mbb{Z}_p}\mbb{Q}_p)
\simeq \mrm{Ind}^{G_{K_m}}_{G_{K'_m}}V_{T'}
\simeq V_T
$
of representations of $G_{K_m}$.
Therefore, we conclude that
Question \ref{question2} has a negative answer for any $n\ge 1$. 

\noindent
(3) This is a special case of  \cite[Corollary 4.20]{Oz}.
\end{proof}

\end{document}